\documentclass[10pt]{amsart}

\usepackage{latexsym,amsfonts,amssymb,mathrsfs,graphicx,subfigure} 
\usepackage{color}
\usepackage[nobysame]{amsrefs}

\theoremstyle{plain}
\newtheorem{theorem}{Theorem}[section]
\newtheorem{lemma}{Lemma}[section]

\newtheorem{coro}[lemma]{Corollary}

\theoremstyle{definition}

\numberwithin{equation}{section}

\theoremstyle{remark}
\newtheorem*{remark}{Remark}

\newcommand{\ade}{\textit{a.e.}}

\newcommand{\RR}{\mathbb{R}}

\newcommand{\wt}[1]{\widetilde{#1}}
\newcommand{\wh}[1]{\widehat{#1}}

\newcommand{\mc}[1]{\mathcal{#1}}

\newcommand{\eps}{\varepsilon}

\newcommand{\abs}[1]{\left\lvert#1\right\rvert}

\newcommand{\Lr}[1]{\left(#1\right)}
\newcommand{\inner}[2]{\left\langle#1,#2\right\rangle}
\newcommand{\diff}[2]{\dfrac{\partial #1}{\partial #2}}
\newcommand{\nm}[2]{\|#1\|_{#2}}
\def\negint{{\int\negthickspace\negthickspace\negthickspace
\negthinspace -}}
\newcommand{\set}[2]{\left\{\,#1\,\mid\,#2\,\right\}}
\def\al{\alpha}

\def\na{\nabla}
\def\pa{\partial}
\def\lam{\lambda}
\def\Lam{\varLambda}

\def\Om{\Omega}
\def\rth{\rightharpoonup}
\def\x{\times}

\def\a{a^{\,\eps}}
\def\bb{b^{\,\eps}}
\def\uu{u^\eps}
\def\vv{v^\eps}
\def\A{\mc{A}}
\def\cu{u_0}

\def\xxe{(x,x/\eps)}

\def\dx{\,\mathrm{d}x}
\def\dy{\,\mathrm{d}y}

\newcommand{\nn}{\nonumber}

\def\div{\operatorname{div}}
\def\HMM{\text{HMM}}
\def\hmm{\emph{\HMM}}

\begin{document}
\title
{A Concurrent
  Global-Local Numerical Method for Multiscale PDEs}
\author[Y.F. Huang]{Yufang Huang} \address{Academy of Mathematics and
  Systems Science, Chinese Academy of Sciences, No. 55, East Road
  Zhong-Guan-Cun, Beijing 100190, China, and School of Mathematical
  Sciences, University of Chinese Academy of Sciences, Beijing 100049,
  China}\email{huangyufang@lsec.cc.ac.cn}

\author[J. Lu]{Jianfeng Lu}
\address{Department of Mathematics, Department of Physics, and Department of Chemistry, Duke University, Box 90320, Durham, NC, 27708 USA}
\email{jianfeng@math.duke.edu}

\author[P.-B. Ming]{Pingbing Ming}
\address{The State Key Laboratory of Scientific and Engineering Computing,
Academy of Mathematics and Systems Science, Chinese Academy of Sciences, No. 55, East Road Zhong-Guan-Cun, Beijing 100190, China, and School of Mathematical Sciences, University of Chinese Academy of Sciences, Beijing 100049, China} \email{mpb@lsec.cc.ac.cn}



\date{\today}

\begin{abstract}
  We present a new hybrid numerical method for multiscale partial
  differential equations, which simultaneously captures the global
  macroscopic information and resolves the local microscopic events
  over regions of relatively small size. The method couples
  concurrently the microscopic coefficients in the region of interest
  with the homogenized coefficients elsewhere. The cost of the method
  is comparable to the heterogeneous multiscale method, while being
  able to recover microscopic information of the solution. The
  convergence of the method is proved for problems with bounded and
  measurable coefficients, while the rate of convergence is
  established for problems with rapidly oscillating periodic or
  almost-periodic coefficients. Numerical results are reported to show
  the efficiency and accuracy of the proposed method.
\end{abstract}
\maketitle

\section{Introduction}
Consider the elliptic problem with Dirichlet boundary condition
\begin{equation}\label{eq:ell}
\left\{\begin{aligned}
-\div\Lr{\a(x)\na\uu(x)}&=f(x),\qquad&& x\in D\subset\RR^n,\\
 \uu(x)&=0,\qquad&& x\in\pa D,
 \end{aligned}\right.
\end{equation}
where $D$ is a bounded domain in $\RR^n$ and $\eps$ is a small
parameter that signifies explicitly the multiscale nature of the
coefficient $\a$. We assume $\a$
belongs to a set $\mc{M}(\lam,\Lam;D)$ that is defined as
\begin{align*}
  \mc{M}(\lam,\Lam;D){:}=\Bigl\{ a\in[L^\infty(D)]^{n\x n}\;\mid\;& \xi \cdot a(x)\xi \ge\lam\abs{\xi}^2,
  \xi \cdot a(x)\xi\ge(1/\Lam)\abs{a(x)\xi}^2\\
 &\quad\text{for any\;}\xi\in\RR^n\text{\;and\;} a.e. \;x\text{\;in\;} D\;\Bigr\},
\end{align*}
where $\abs{\cdot}$ denotes the Euclidean norm in $\RR^n$. Note that
the coefficients $\a$ in $\mc{M}(\lam, \Lam;D)$ are not necessarily
symmetric.

The large scale behavior of the solution of~\eqref{eq:ell} is well
understood by the theory of homogenization. In the sense of H-convergence
due to {\sc Murat and Tartar}~\cite[Theorem 6.5]{Tartar:2009}
and~\cite{MuratTartar:1997}, for every $\a\in\mc{M}(\lam,\Lam;D)$ and
$f\in H^{-1}(D)$ the sequence of solutions $\{\uu\}$ of~\eqref{eq:ell}
satisfies
\[
\begin {aligned}
\uu&\rth\cu\qquad&&\text{weakly in}\quad H_0^1(D), \, \text{and}\\
\a\na\uu&\rth\A\na\cu\qquad&&\text{weakly in}\quad [L^2(D)]^n,
\end {aligned}
\]
where $\cu$ is the solution of a homogenized problem
\begin{equation}\label{homoell}
\left\{\begin{aligned}
-\div\Lr{\A(x)\na\cu(x)}&=f(x),\qquad&& x\in D,\\
 \cu(x)&=0,\qquad&& x\in\pa D,
 \end{aligned}\right.
\end{equation}
with the homogenized coefficient $\A\in\mc{M}(\lam,\Lam;D)$.

For multiscale PDEs as \eqref{eq:ell}, the quantities of interest
include the macroscopic behavior of the solution and also the
microscopic information (local fluctuation) of the
solution~\cite{Ebook:2011}.  Many numerical approaches based on the
idea of homogenization have been proposed and thoroughly studied in
the literature,
such as the multiscale finite element method~\cite{HouWu:1997} and the
heterogeneous multiscale method (HMM)~\cite{EEnquist:2003}.

In this work, our focus is the scenario where the microscopic
coefficient $\a$ is only available in part of the domain, while
outside the region, only a coarse information is available about the
coefficient field. More specifically, we only assume the knowledge of
the homogenized coefficients outside a small region of the domain. The
question is that given this information, whether it is still possible
to recover the macroscopic behavior of the solution, together with
resolving the local fluctuation of the solution, where the detailed
information of the coefficient is known.

Several numerical approaches have been developed in recent years for
such scenario. Those methods can be roughly put into two categories.

The first class is the \emph{global-local} approach firstly proposed
in~\cites{OdenVemaganti:2000, OdenVemaganti:2001}, and further
developed in~\cites{EEnquist:2003, EMingZhang:2005, BabuskaLipton:2011,
BabuskaMotamedTempone:2014}. This is a two stage method: One first
computes the homogenized {\em global} solution $\cu$ over the whole
domain and then one finds the {\em local} fluctuation by solving an
extra problem on a local part of the domain. The homogenized solution
may be used as boundary condition in solving the local problem or be
used to provide information on the fine scale based on a
$L^2$-projection.  The convergence of this approach has been
investigated numerically in~\cite{MingYue:2006} for problems with many
scales, within the HMM framework~\cites{HMMreview1,HMMreview2}. This approach has also been
critically reviewed in~\cite{BabuskaLiptonStuebner:2008}, where in
particular the choice of the local approximation space was
investigated. 

Another class of method is based on the idea of domain decomposition,
which concerns handshaking multiple operators acting on different
parts of the physical domain. Those operators may be either the
restrictions of the same governing differential operators to the
overlapping or non-overlapping sub-domains~\cite{DuGunzburger:2000,
  Gervasio:2001, Glowinski:04, Pironneau:11}, or different
differential operators that describe perhaps different physical
laws~\cite{KuberryLee:2013, Quarterior:2014}.  The popular Arlequin
method~\cite{BenDhia:1998, BenDhia:2005} also belongs to this
category, for which the agreement of solutions on different scales is
enforced using a Lagrange multiplier approach. A more recent work
is~\cite{Abdulle:16}, which considers a method following the
optimization-based coupling strategy \cite{DuGunzburger:2000,
  Gervasio:2001}: A discontinuous Galerkin HMM is used in a region
with scale separation (periodic media), while a standard continuous
finite element method is used in a region without scale separation,
the unknown boundary conditions at the interface are supplied by
minimizing the difference between the solutions in the overlapped
domain. The well-posedness and the convergence of the method have been
proved, while the convergence rate is yet unknown.


In this contribution, we propose a new concurrent global-local
method to capture both the average information and the local
microscale information simultaneously, as we shall explain in more
details below. The current approach is mainly inspired by the recent
work~\cite{LuMing:2011, LuMing:2014} by two of the authors, in which a
hybrid method that couples force balance equations from the atomistic
model and the Cauchy-Born elasticity is proposed and
analyzed. Such method is proven to have sharp stability and optimal
convergence rate.

Compared with the sequential global-local approach, our proposed
method is a concurrent approach. Compared with the domain
decomposition approach, our proposed method smoothly blends together
the fine scale and coarse scale problem, instead of the usual coupling
in domain-decomposition approach via boundary conditions or volumetric
matching. To some extent, our coupling strategy can be understood as
directly enforcing the agreement of the solutions at different scales
in the coupling region, rather through the use of a penalty.

More concretely, our method starts with a hybridization of microscopic
and macroscopic coefficients as follows. For a transition function
$\rho$ satisfying $0\le\rho\le 1$, we define the hybrid coefficient as
\begin{equation}\label{eq:hybridcoef}
\bb(x){:}=\rho(x)\a(x)+(1-\rho(x))\A(x).
\end{equation}
Note in particular that $\a$ is only needed where $\rho \neq 0$, and
only the homogenized coefficient $\A$ is used outside the support of
$\rho$. This viewpoint is particularly useful when the microscopic
information of the elliptic coefficient is not accessible everywhere.

On the continuum level, we solve the following problem with the hybrid
coefficient $\bb$: find $\vv\in H_0^1(D)$ such that
\begin{equation}\label{eq:varaarlequin}
\inner{\bb\na\vv}{\na w}=\inner{f}{w}\qquad\text{for all\quad}w\in H_0^1(D),
\end{equation}
where we denote the $L^2(D)$ inner product by $\inner{\cdot}{\cdot}$,
and the $L^2(\wt{D})$ inner product by
$\inner{\cdot}{\cdot}_{L^2(\wt{D})}$ for any measurable subset
$\wt{D}\subset D$.
It is clear that $\bb\in\mc{M}(\lam,\Lam;D)$, and the existence and uniqueness of the solution of
Problem~\eqref{eq:varaarlequin} follows from the Lax-Milgram theorem.

To numerically solve~\eqref{eq:varaarlequin}, let $X_h\subset H_0^1(D)$ be a standard Lagrange finite element space consisting of piecewise polynomials
of degree $r-1$, we find
$v_h\in X_h$ such that
\begin{equation}\label{eq:varaarlequinapp}
\inner{\bb_h\na v_h}{\na w}=\inner{f}{w}\qquad\text{for all\quad}w\in X_h,
\end{equation}
where
\[
\bb_h=\rho(x)\a(x)+(1-\rho(x))\A_h(x),
\]
and $\A_h$ is an approximation of $\A$. In practice, if the
homogenized coefficient is not directly given, $\A_h$ may be obtained
by HMM type method or any other numerical homogenization / upscaling
approaches. For practical concerns, we assume that the support of
$\rho$ is small, which means that we essentially solve the homogenized
problem in the most part of the underlying domain, where
$\rho\simeq 0$, while the original problem is solved wherever the
microscale information is of particular interest, where
$\rho\simeq 1$. The goal is to get the microscopic information
together with the macroscopic behavior with computational cost
comparable to solving the homogenized equation.

Note that $\bb\na\vv=\rho\a\na\vv+(1-\rho)\A\na\vv$ is a hybrid flux
(i.e., a hybrid stress tensor for elasticity problem), which reads
\[
\bb\na\vv=\left\{\begin{aligned}
& \a\na\vv, \quad&&\text{if\quad}\rho(x)=1,\\
& \A\na\vv, \quad&&\text{if\quad}\rho(x)=0,\\
& \rho\a\na\vv+(1-\rho)\A\na\vv,\quad&&\text{otherwise.}
\end{aligned}\right.
\]
This implies that the proposed hybrid method actually mixes the
flux/stress in a weak sense, which is different to the approach
in~\cite{LuMing:2011, LuMing:2014} that mixes the forces in a strong
sense. This is more appropriate because Problem~\eqref{eq:ell} is in
divergence form. It is perhaps worth pointing out that the proposed
method differs from the well-known partition of unit
method~\cite{BabuskaMelenk:1997}, which incorporate the partition of
unit function into the approximating space while we directly blend the
differential operators on the continuum level by the transition functions.

We emphasize that the working assumption is that the microscopic
information is only desired on a region with relatively small size,
which might lie in the interior or possibly near the boundary of the
whole domain, or even abut the boundary of the domain. Outside the
part where the oscillation is resolved, we could at best hope for
capturing the macroscopic information of the solution.  This motivates
that we should only expect the convergence of the proposed method to
the microscopic solution $\uu$ in a {\em local} energy norm instead of
a global norm. Moreover, such local energy estimate should allow for
highly refined grid that is quite often in practice, otherwise, the
local events cannot be resolved properly.

The structure of the paper is as follows. In \S~\ref{sec:hlimit}, we
study the H-limit of the hybrid method without taking into account the
discretization.  In \S~\ref{sec:errdis}, the error estimate of the
proposed method with discretization is proved, in particular, the
local energy error estimate is established over a highly refined grid,
which is the main theoretical result of this paper. In the last section, we report some numerical examples that validate the method. In the Appendix, we construct a one-dimensional
example to show the size-dependence of the estimate
over the measure of the support of the transition function $\rho$.

Throughout this paper, we shall use standard notations for Sobolev spaces, norms and seminorms, cf.,~\cite{AdamsFournier:2003}, e.g.,
\[
\|u\|_{H^1(D)}{:}=\Lr{\int_D(u^2+\abs{\na u}^2)\dx}^{1/2},\quad
\abs{u}_{W^{k,p}(D)}{:}=\Lr{\sum_{\abs{\alpha}=k}\nm{D^{\alpha}u}{L^p(D)}^p}^{1/p}.
\]
We use $C$ as a generic constant independent of $\eps$ and the mesh size
$h$, which may change from line to line.
\section{H-Convergence of the Concurrent Method}\label{sec:hlimit}
Before considering the convergence of the method, we first study the
implication of the strategy of mixing microscopic and
  homogenized coefficients together as \eqref{eq:hybridcoef}. To
separate the influence of the discretization, we consider in this
section the continuous Problem~\eqref{eq:varaarlequin}. The
discretized problem is studied in the next section. By
H-convergence theory, there exists a matrix
$\mc{B}\in\mc{M}(\lam,\Lam;D)$ that is the H-limit of
$\bb$. The following theorem quantifies the difference between
$\A$ and $\mc{B}$.
\begin{theorem}\label{thm:hlimit}
There holds
\begin{equation}\label{eq:perturb}
\|\A(x)-\mc{B}(x)\|\le 2\Lam\Lr{\Lam/\lam+\sqrt{\Lam/\lam}}\rho(x)(1-\rho(x))\quad\ade\; x\in D.
\end{equation}
Here $\|\cdot\|$ is the Frobenius norm of a matrix.
\end{theorem}

It follows from the above result that $\mc{B}\equiv\mc{A}$ whenever
$\rho(x)=0$ or $\rho(x)=1$, a.e., $x\in D$, which fits the intuition.
When $0<\rho<1$, the above estimate gives a quantitative estimate
about the distance between the effective matrices $\mc{B}$ and
$\mc{A}$.

The proof is based on a perturbation result of H-limit, which can be stated
as the following lemma in terms of our notation.
\begin{lemma}~\cite[Lemma
10.9]{Tartar:2009}\label{lema:tartar}
If $\a\in\mc{M}(\al,\beta;D)$ and $\bb\in\mc{M}(\al',\beta';D)$ H-converges to $\mc{A}$ and $\mc{B}$,
and $\|\a(x)-\bb(x)\|\le\epsilon$ for a.e.~$x\in D$, then
\begin{equation}\label{eq:hlimitperturbsym}
\max_{x\in D}\|\mc{A}(x)-\mc{B}(x)\|\le\epsilon\sqrt{\dfrac{\beta\beta'}{\al\al'}}.
\end{equation}
\end{lemma}

We shall not directly use Lemma~\ref{lema:tartar},
while our proof largely follows the idea of the
proof of this lemma.  \vskip .5cm

\noindent{\em Proof of Theorem~\ref{thm:hlimit}\;}
Firstly, we prove
\begin{equation}\label{eq:matrixrela1}
\|\A(x)-\mc{B}(x)\|\le\dfrac{2\Lam^2}{\lam}(1-\rho(x))\quad\ade\; x\in D.
\end{equation}

For any $f,g\in H^{-1}(D)$, we solve
\[
\left\{\begin{aligned}
-\div(\a\na\phi^\eps)&=f,\qquad&&\text{in\quad}D,\\
\phi^\eps&=0,\qquad&&\text{on\quad}\pa D,
\end{aligned}\right.
\]
and
\[
\left\{\begin{aligned}
    -\div\Lr{(\bb)^t\na\psi^\eps}&=g,\qquad&&\text{in\quad}D,\\
    \psi^\eps&=0,\qquad&&\text{on\quad}\pa D,
\end{aligned}\right.
\]
where $(\bb)^t$ is the transpose of the matrix $\bb$.
By H-limit theorem~\cite[Theorem 6.5]{Tartar:2009}, there exist
$\A,\mc{B}\in\mc{M}(\lam,\Lam;D)$ such that
\[\left\{\begin{aligned}
\phi^\eps&\rth\phi_0\qquad&&\text{weakly in}\quad H_0^1(D),\;\text{and}\\
\a\na\phi^\eps&\rth\A\na\phi_0\qquad&&\text{weakly in}\quad [L^2(D)]^n,
\end{aligned}\right.
\]
and
\[
\left\{
\begin{aligned}
\psi^\eps&\rth\psi_0\qquad&&\text{weakly in}\quad H_0^1(D),\;\text{and}\\
(\bb)^t\na\psi^\eps&\rth(\mc{B})^t\na\psi_0\qquad&&\text{weakly in}\quad [L^2(D)]^n,
\end{aligned}\right.
\]
with
\[
\left\{\begin{aligned}
-\div(\A\na\phi_0)&=f,\qquad&&\text{in\quad}D,\\
\phi_0&=0,\qquad&&\text{on\quad}\pa D,
\end{aligned}\right.
\]
and
\[
\left\{\begin{aligned}
-\div\Lr{\mc{B}^t\na\psi_0}&=g,\qquad&&\text{in\quad}D,\\
\psi_0&=0,\qquad&&\text{on\quad}\pa D.
\end{aligned}\right.
\]

By the {\em Div-Curl Lemma}~\cite{Tartar:1979}, we conclude
\begin{equation}\label{eq:divcurl}
\left\{\begin{aligned}
\inner{\a\na\phi^\eps}{\na\psi^\eps}&\to
\inner{\A\na\phi_0}{\na\psi_0}\quad&&\text{in the sense of measure},\\
\inner{(\bb)^t\na\psi^\eps}{\na\phi^\eps}&\to
\inner{(\mc{B})^t\na\psi_0}{\na\phi_0}\quad&&\text{in the sense of measure}.
\end{aligned}\right.
\end{equation}
Therefore, for any $\varphi\in C_0^\infty(D)$, we have
\[
\lim_{\eps\to 0}\inner{\varphi(\bb-\a)\na\phi^\eps}{\na\psi^\eps}
\to\inner{\varphi(\mc{B}-\A)\na\phi_0}{\na\psi_0}.
\]
Let $\varphi\ge 0$, and we define
\[
X{:}=\lim_{\eps\to 0}\inner{\varphi(\bb-\a)\na\phi^\eps}{\na\psi^\eps}.
\]
It is clear that
\begin{align*}
X&\le\limsup_{\eps\to 0}\inner{\varphi(1-\rho)\abs{(\A-\a)\na\phi^\eps}}{\abs{\na\psi^\eps}}\\
&\le 2\Lam\limsup_{\eps\to 0}\inner{\varphi(1-\rho)\abs{\na\phi^\eps}}{\abs{\na\psi^\eps}}.
\end{align*}
For any $\al>0$, we bound $X$ as
\begin{align*}
X&\le\dfrac{2\Lam}{\lam}\Lr{\al\lam\limsup_{\eps\to 0}\inner{\varphi(1-\rho)}{\abs{\na\phi^\eps}^2}+\dfrac{\lam}{4\al}\limsup_{\eps\to 0}\inner{\varphi(1-\rho)}{\abs{\na\psi^\eps}^2}}\\
&\le\dfrac{2\Lam}{\lam}\Bigl(\al\limsup_{\eps\to 0}\inner{\varphi(1-\rho)\a\na\phi^\eps}{\na\phi^\eps}+\dfrac1{4\al}\limsup_{\eps\to 0}\inner{\varphi(1-\rho)\bb\na\psi^\eps}{\na\psi^\eps}\Bigr).
\end{align*}
Invoking the {\em Div-Curl Lemma}~\eqref{eq:divcurl} once again, we obtain
\begin{align*}
X&\le\dfrac{2\Lam}{\lam}\Lr{\al
\inner{\varphi(1-\rho)\A\na\phi_0}{\na\phi_0}
+\dfrac1{4\al}\inner{\varphi(1-\rho)\mc{B}\na\psi_0}{\na\psi_0}}\\
&\le\dfrac{2\Lam^2}{\lam}\Lr{\al
\inner{\varphi(1-\rho)}{\abs{\na\phi_0}^2}
+\dfrac1{4\al}\inner{\varphi(1-\rho)}{\abs{\na\psi_0}^2}},
\end{align*}
which implies that for a.e. $x\in D$,
\[
\abs{(\mc{B}-\A)\na\phi_0\cdot\na\psi_0}\le\dfrac{2\Lam^2}{\lam}\bigl(1-\rho(x)\bigr)
\Lr{\al\abs{\na\phi_0}^2+\dfrac1{4\al}\abs{\na\psi_0}^2}.
\]
Optimizing $\al$, we obtain that for a.e. $x\in D$,
\[
\abs{(\mc{B}-\A)\na\phi_0\cdot\na\psi_0}\le\dfrac{2\Lam^2}{\lam}\bigl(1-\rho(x)\bigr)
\abs{\na\phi_0}\abs{\na\psi_0},
\]
from which we obtain~\eqref{eq:matrixrela1} because $\phi_0$ and $\psi_0$ are arbitrary.

Next, we prove
\begin{equation}\label{eq:matrixrela2}
\|\A(x)-\mc{B}(x)\|\le 2\Lam\sqrt{\Lam/\lam}\,\rho(x)\quad\ade\; x\in D.
\end{equation}
The proof of~\eqref{eq:matrixrela2} is essentially the same with the one that leads to~\eqref{eq:matrixrela1} except that we define
\[
X{:}=\lim_{\eps\to 0}\inner{\varphi(\bb-\A)\na\phi_0}{\na\psi^\eps}\qquad\text{for\quad}\varphi\ge 0.
\]
It is clear that
\[
X\le\limsup_{\eps\to 0}\inner{\varphi\rho\abs{(\A-\a)\na\phi_0}}{\abs{\na\psi^\eps}}\le 2\Lam\limsup_{\eps\to 0}\inner{\varphi\rho\abs{\na\phi_0}}{\abs{\na\psi^\eps}}.
\]
For any $\al>0$, we bound $X$ as
\begin{align*}
X&\le 2\Lam\Lr{\al\inner{\varphi\rho}{\abs{\na\phi_0}^2}
+\dfrac1{4\al}\limsup_{\eps\to 0}\inner{\varphi\rho}{\abs{\na\psi^\eps}^2}}\\
&\le 2\Lam\Lr{\al\inner{\varphi\rho}{\abs{\na\phi_0}^2}
+\dfrac1{4\al\lam}\limsup_{\eps\to 0}\inner{\varphi\rho\bb\na\psi^\eps}{\na\psi^\eps}}.
\end{align*}
Applying the {\em Div-Curl Lemma}~\eqref{eq:divcurl} to the second term in the right-hand side of the above
equation, we obtain
\begin{align*}
X&\le 2\Lam\Lr{\al
\inner{\varphi\rho\na\phi_0}{\na\phi_0}
+\dfrac1{4\al\lam}\inner{\varphi\rho\mc{B}\na\psi_0}{\na\psi_0}}\\
&\le 2\Lam\Lr{\al
\inner{\varphi\rho}{\abs{\na\phi_0}^2}
+\dfrac{\Lam}{4\al\lam}\inner{\varphi\rho}{\abs{\na\psi_0}^2}},
\end{align*}
which implies that for a.e. $x\in D$,
\[
\abs{(\mc{B}-\A)\na\phi_0\cdot\na\psi_0}\le 2\Lam\rho
\Lr{\al\abs{\na\phi_0}^2+\dfrac{\Lam}{4\al\lam}\abs{\na\psi_0}^2}.
\]
Optimizing $\al$, we obtain
\[
\abs{(\mc{B}-\A)\na\phi_0\cdot\na\psi_0}\le 2\Lam\sqrt{\Lam/\lam}\,\rho(x)
\abs{\na\phi_0}\abs{\na\psi_0},
\]
from which we obtain~\eqref{eq:matrixrela2}.

Finally we use the convex combination of~\eqref{eq:matrixrela1} and~\eqref{eq:matrixrela2} as
\[
\|\A(x)-\mc{B}(x)\|=\rho(x)\|\A(x)-\mc{B}(x)\|+\Lr{1-\rho(x)}\|\A(x)-\mc{B}(x)\|
\]
for $\ade\;x\in D$, this leads to~\eqref{eq:perturb}.
\qed

If we replace $\mc{A}$ by any matrix $\mc{C}\in\mc{M}(\lam',\Lam';D)$, then we may slightly generalize the above theorem as
\begin{coro}\label{coro:hlimit}
Let $\mc{C}\in\mc{M}(\lam',\Lam';D)$, and we define $b^\eps=\rho(x)\a(x)+(1-\rho(x))\mc{C}$. Denote by $\mc{B}$ the H-limit of $b^\eps$, then for $\ade\;x\in D$, there holds
\[
\|\mc{B}(x)-\rho(x)\mc{A}(x)-(1-\rho(x))\mc{C}(x)\|\le\Lr{\Lam+\wt{\Lam}}\sqrt{\wt{\Lam}/\wt{\lam}}\Lr{\sqrt{\Lam/\lam}+1}\rho(x)(1-\rho(x)),
\]
where $\wt{\Lam}=\Lam\vee\Lam'$ and $\wt{\lam}=\lam\wedge\lam'$.
\end{coro}
The proof is omitted because it follows essentially the same line that
leads to Theorem~\ref{thm:hlimit}.

As a direct consequence of the above corollary, if we take $\rho(x)$ as the characteristic function of a subdomain $D_0$ of $D$, i.e., $\rho=\chi_{D_0}$, then
\begin{equation}\label{eq:exact1}
\mc{B}(x)=\chi_{D_0}\mc{A}(x)+(1-\chi_{D_0})\mc{C}(x)\qquad\ade\;x\in D.
\end{equation}
In particular, if we take $\mc{C}(x)=\mc{A}(x)$, then
\begin{equation}\label{eq:exact2}
\mc{B}(x)=\mc{A}(x)\qquad\ade\;x\in D.
\end{equation}

When $\a$ is locally periodic, i.e., $\a=a\xxe$ with
$a(x,\cdot)$ is $Y$-periodic with $Y=(-1/2,1/2)^n$, we can
characterize the effective matrix $\mc{B}$ more explicitly since $\bb$
is also locally periodic with the same period. By classical
homogenization theory~\cite{BenssousonLionsPapanicalou:1978}, the
effective matrix $\mc{B}$ is given by
\begin{equation}\label{eq:effective}
\mc{B}_{ij}(x)=\negint_Y\Lr{b_{ij}+b_{ik}\diff{\chi_{\rho}^j}{y_k}}(x,y)\dy,
\end{equation}
where $\{\chi_{\rho}^j(x,y)\}_{j=1}^d$ is periodic in
$y$ with period $Y$ and it satisfies
\begin{equation}\label{eq:cell}
-\diff{}{y_i}\Lr{b_{ik}\diff{\chi_{\rho}^j}{y_k}}(x,y)
= \diff{b_{ij}}{y_i}(x,y)\quad\text{in\quad}Y,\qquad
\int_Y\chi_\rho^j(x,y)\dy=0.
\end{equation}
For $x\in D$ with $\rho(x)=0$ or $\rho(x)=1$, we have $\mc{B}=\mc{A}$.

In particular, for $n=1$, we have the following explicit formula for $\mc{B}$.
\[
\mc{B}(x)=\Lr{\int_0^1\dfrac{1}{\rho(x) a(x,y)+(1-\rho(x))\mc{A}(x)}\dy}^{-1},
\]
where
\[
\mc{A}(x)=\Lr{\int_0^1\dfrac1{a(x,y)}\dy}^{-1}.
\]
As expected,  when $\rho(x)=0$ or $\rho(x)=1$, it is clear from the above that $\mc{B}(x)=\mc{A}(x)$.
\section{Convergence Rate for the Discrete Problem}\label{sec:errdis}
We now study the convergence rate of the discrete
Problem~\eqref{eq:varaarlequinapp}. We assume that
$\mc{A}_h\in\mc{M}(\lam',\Lam';D)$. This is true for any reasonable
approximation of $\mc{A}$. For example, if we use HMM method
\cite{EMingZhang:2005, HMMreview1,HMMreview2} to compute the effective
matrix, then $\mc{A}_h\in\mc{M}(\lam,\Lam;D)$. By this assumption, we have
\(\bb_h\in\mc{M}(\lam,\Lam;D)\).

To step further, let $\mc{T}_h$ be a triangulation of $D$ with maximum mesh size $h$. Denote by $h_{\tau}$ the diameter of each element $\tau\in\mc{T}_h$. we assume that all
elements in $\mc{T}_h$ are shape-regular in the sense of Ciarlet and Raviart~\cite{Ciarlet:1978}, that is each $\tau\in\mc{T}_h$ contains a ball of radius $c_1h_{\tau}$
and is contained in a ball of radius $C_1 h_{\tau}$ with fixed constants $c_1$ and $C_1$. 


Denote $K=\text{supp}\;\rho$ and $\abs{K}{:}=\text{mes}K$, and define
\[
\eta(K)=
\begin{cases}
  \abs{\ln\abs{K}}^{1/2}, &\text{if } n=2\;\text{and}\;s=1,\\
1, & \text{if } n = 3\; \text{or}\; s \in (0, 1).
\end{cases}
\]

We begin with the following inequality that will be frequently used later on.
\begin{lemma}
For any $v\in H^s(D)$ with $s\in(0,1]$, and for any subset $\Om\subset D$, we have
\begin{equation}\label{eq:basicest}
\nm{v}{L^2(\Om)}\le C\abs{\Om}^{s/n}\eta(\Om)\nm{v}{H^s(D)},
\end{equation}
where the constant $C$ independent of the measure of $D$.
\end{lemma}

\begin{proof}
For $n=3$ with $0<s\le 1$ and $n=2$ with $0<s<1$, let $2^\ast=2n/(n-2s)$ be the fractional critical exponent. By the H\"older's
inequality and the Sobolev embedding inequality~\cite{Nezza:2012}, we obtain
\[
\nm{v}{L^2(\Om)}\le\abs{\Om}^{1/2-1/2^\ast}\nm{v}{L^{2^\ast}(\Om)}\le\abs{\Om}^{s/n}\nm{v}{L^{2^\ast}(D)}\le
C\abs{\Om}^{s/n}\nm{v}{H^s(D)},
\]
which yields~\eqref{eq:basicest}.

As to $n=2$ and $s=1$, for any $p>2$, we have the Sobolev embedding inequality
\[
\nm{v}{L^p(D)}\le C\sqrt{p}\nm{v}{H^1(D)}\qquad\text{for all\quad}v\in H^1(D),
\]
which together with the H\"older's inequality gives
\[
\nm{v}{L^2(\Om)}\le\abs{\Om}^{1/2-1/p}\nm{v}{L^p(\Om)}\le\abs{\Om}^{1/2-1/p}\nm{v}{L^p(D)}\le C\sqrt{p}\abs{\Om}^{1/2-1/p}\nm{v}{H^1(D)}.
\]
Taking $p=\abs{\ln\abs{\Om}}$ in the above inequality, we obtain~\eqref{eq:basicest} for  $n=2$ and $s = 1$.
\end{proof}
\begin{remark}\label{remark:1dineq}
When $n=s=1$, \eqref{eq:basicest} is still true with prefactor replaced by $\abs{\Om}^{1/2}$ by observing
\[
\nm{v}{L^2(\Om)}\le\abs{\Om}^{1/2}\nm{v}{L^\infty(\Om)}\le\abs{\Om}^{1/2}\nm{v}{L^\infty(D)}\le C\abs{\Om}^{1/2}\nm{v}{H^1(D)},
\]
where we have used the imbedding $H^1(D)\hookrightarrow L^\infty(D)$ in the last step.
\end{remark}
\subsection{Accuracy for retrieving the macroscopic information}
In this part, we estimate the approximation error between the hybrid solution and the homogenized solution.
The following result is in the same spirit of the {\sc first Strang lemma}~\cite{Ciarlet:1978}. Our proof relies on the {\em Meyers' regularity} result~\cite {Meyers:1963}
for Problem~\eqref{homoell} in an essential way. We state Meyers' results as follows. There exists $p_0>2$ that depends on $D,\Lam$ and $\lam$, such that for all $p\le p_0$,
\begin{equation}\label{eq:meyers}
\nm{\na\cu}{L^p(D)}\le C\nm{f}{W^{-1,p}(D)}
\end{equation}
with $C$ depends on $D,\Lam$ and $\lam$.

For any $\Om\subset D$, by H\"older inequality and the above Meyers' estimate, we obtain
\begin{align}\label{eq:applmeyers}
\nm{\na\cu}{L^2(\Om)}&\le\abs{\Om}^{1/2-1/p}\nm{\na\cu}{L^p(\Om)}\le\abs{\Om}^{1/2-1/p}\nm{\na\cu}{L^p(D)}\\\nn
&\le C\abs{\Om}^{1/2-1/p}\nm{f}{W^{-1,p}(D)},
\end{align}
where $C$ depends on $D$ but independent of $\Om$.
\begin{theorem}\label{thm:disenerhomo}
Let $\cu$ and $v_h$ be the solutions of Problem~\eqref{homoell} and Problem~\eqref{eq:varaarlequinapp},
respectively. Define $e(\hmm){:}=\max_{x\in D\setminus K}\|(\A-\A_h)(x)\|$. There exists $C$ depends on $\lam,\Lam$ and $D$ such that
for any $2<p<p_0$,
\begin{equation}\label{eq:diserrhomo1}
\begin{aligned}
\nm{\na(\cu-v_h)}{L^2(D)}&\le C\inf_{\chi\in X_h}\nm{\na(\cu-\chi)}{L^2(D)}\\
&\qquad+C\Lr{\abs{K}^{1/2-1/p}+e(\hmm)}\nm{f}{W^{-1,p}(D)},
\end{aligned}
\end{equation}
and for any $2<p<\wt{p}_0$ with
\[
\wt{p}_0=\left\{\begin{aligned}
p_0,\quad&\text{if \;}n=2,\\
p_0\wedge 6,\quad&\text{if\;}n=3.
\end{aligned}\right.
\]
we have
\begin{multline}\label{eq:diserrhomo2}
\nm{\cu-v_h}{L^2(D)} \le C\Lr{\inf_{\chi\in X_h}\nm{\na(\cu-\chi)}{L^2(D)}+\Bigl(\abs{K}^{1/2-1/p}+e(\hmm)\Bigr)\nm{f}{W^{-1,p}(D)}}\\
\times\Lr{\sup_{g\in L^2(D)}\dfrac{1}{\nm{g}{L^2(D)}}
\inf_{\chi\in X_h}\nm{\na(\varphi_g-\chi)}{L^2(D)}+\abs{K}^{1/2-1/p}},
\end{multline}
where $\varphi_g\in H_0^1(D)$ is the unique solution of the variational problem:
\begin{equation}\label{eq:auxvara}
\inner{\A\na v}{\na\varphi_g}=\inner{g}{v}\qquad\text{for all\quad}v\in H_0^1(D).
\end{equation}
\end{theorem}

The above estimates show that the solution of the hybrid problem is a
good approximation of the solution of the homogenized problem provided
that $\abs{K}$ is small, besides certain approximation error.  This is
expected since in this case we essentially solve the homogenized
problem over the main part of the domain $D$. The dependence on
$\abs{K}$ in the estimate~\eqref{eq:diserrhomo1} is also essential and
sharp, as will be shown by an explicit one-dimensional example in the
Appendix~\ref{subsec:example}. Similar constructions can be also done
for higher dimensions, though it becomes much more tedious.

Let us also remark that the error estimates~\eqref{eq:diserrhomo1}
and~\eqref{eq:diserrhomo2} are valid {\em without} any smoothness
assumption on $\cu$. Convergence rate may be obtained if we assume
extra smoothness on $\cu$, which may be found in
Corollary~\ref{coro:expliciterr}.  \vskip .5cm\noindent{\em Proof of
  Theorem~\ref{thm:disenerhomo}\;} Let $\wt{u}\in X_h$ be the solution
of the variational problem
\[
\inner{\A\na\wt{u}}{\na v}=\inner{f}{v}\qquad\text{for all\quad} v\in X_h.
\]
By Cea's lemma~\cite{Ciarlet:1978}, we obtain
\begin{equation}\label{eq:galerkinproj}
\nm{\na(\cu-\wt{u})}{L^2(D)}\le\dfrac{\Lam}{\lam}\inf_{\chi\in X_h}\nm{\na(\cu-\chi)}{L^2(D)}.
\end{equation}
Denote $w=\wt{u}-v_h$ and using the definition of $v_h$ and $\wt{u}$, we obtain
\[
\inner{\bb_h\na w}{\na w}=\inner{\bb_h\na\wt{u}}{\na w}-\inner{f}{w}=\inner{(\bb_h-\A)\na\wt{u}}{\na w}.
\]
By $\bb_h-\A=\rho(\a-\A)+(1-\rho)(\A_h-\A)$, we obtain
\[
\inner{\bb_h\na w}{\na w}\le\Lr{2\Lam\nm{\na\wt{u}}{L^2(K)}+e(\HMM)\nm{\na\wt{u}}{L^2(D)}}\nm{\na w}{L^2(D)}.
\]
Using the triangle inequality and~\eqref{eq:applmeyers} with $\Om=K$, we obtain, for any $2<p<p_0$,
\begin{equation}\label{eq:basicest2}
\nm{\na\wt{u}}{L^2(K)}\le\nm{\na(\cu-\wt{u})}{L^2(D)}+C\abs{K}^{1/2-1/p}\nm{f}{W^{-1,p}(D)},
\end{equation}
and the a-priori estimate
\(
\nm{\na\wt{u}}{L^2(D)}\le C\nm{f}{H^{-1}(D)}\le C\nm{f}{W^{-1,p}(D)}.
\)
Combining the above three equations, we obtain
\begin{align*}
\inner{\bb_h\na w}{\na w}&\le C\nm{\na(\cu-\wt{u})}{L^2(D)}\nm{\na w}{L^2(D)}\\
&\quad+\Lr{\abs{K}^{1/2-1/p}+e(\HMM)}\nm{f}{W^{-1,p}(D)}\nm{\na w}{L^2(D)}.
\end{align*}
This together with~\eqref{eq:galerkinproj} concludes the estimate~\eqref{eq:diserrhomo1}.

We exploit {\em Aubin-Nitsche's dual argument}~\cite{Nitsche:1968} to prove the $L^2$ estimate. For any $\chi\in X_h$, using~\eqref{eq:auxvara}, we obtain
\begin{equation}
\inner{g}{\cu-v_h}=\inner{\A\na(\cu-v_h)}{\na(\varphi_g-\chi)}+\inner{\A\na(\cu-v_h)}{\na\chi}.\label{eq:disl2dec}
\end{equation}
The first term may be bounded by
\[
\abs{\inner{\A\na(\cu-v_h)}{\na(\varphi_g-\chi)}}\le\Lam\nm{\na(\cu-v_h)}{L^2(D)}
\nm{\na(\varphi_g-\chi)}{L^2(D)}.
\]

The second term in the right-hand side of~\eqref{eq:disl2dec} may be rewritten into
\[
\inner{\A\na(\cu-v_h)}{\na\chi}=\inner{f}{\chi}-\inner{\A\na v_h}{\na\chi}=\inner{(\bb_h-\A)\na v_h}{\na\chi}.
\]
By
\(
\bb_h-\A=\rho(\a-\A)+(1-\rho)(\A_h-\A),
\)
we obtain
\begin{align*}
\abs{\inner{\A\na(\cu-v_h)}{\na\chi}}&\le 2\Lam\nm{\na v_h}{L^2(K)}\nm{\na\chi}{L^2(K)}\\
&\quad+e(\HMM)\nm{\na v_h}{L^2(D)}\nm{\na\chi}{L^2(D)}\Bigr).
\end{align*}
Proceeding along the same line that leads to~\eqref{eq:basicest2}, we obtain
\[
\nm{\na v_h}{L^2(K)}\le \nm{\na(\cu-v_h)}{L^2(D)}+C\abs{K}^{1/2-1/p}\nm{f}{W^{-1,p}(D)}.
\]
Proceeding along the same line and using the imbedding $L^2(D)\hookrightarrow W^{-1,p}(D)$ for any $2<p<\wt{p}_0$,
we obtain
\[
\nm{\na\varphi_g}{L^2(K)}\le C\abs{K}^{1/2-1/p}\nm{g}{L^2(D)},
\]
hence
\[
\nm{\na\chi}{L^2(K)}\le\nm{\na(\varphi_g-\chi)}{L^2(D)}+C\abs{K}^{1/2-1/p}\nm{g}{L^2(D)}.
\]
Summing up all the above estimates, using the triangle inequality and~\eqref{eq:diserrhomo1}, we obtain~\eqref{eq:diserrhomo2}.
\qed

In Theorem~\ref{thm:disenerhomo}, the factor $\abs{K}^{1/2-1/p}$ may seem quite pessimistic. If there are some extra conditions on the solution $\cu$ or the source term $f$, the error bound  may be significantly improved.
\begin{coro}\label{coro:expliciterr}
\begin{enumerate}
\item If $\nm{\na\cu}{L^\infty(K)}$ is bounded, then the index $p$ may be infinity.

\item If there holds the regularity estimate
\begin{equation}\label{eq:ellreg}
\nm{\cu}{H^{1+s}(D)}\le C\nm{f}{L^2(D)}\quad\text{for\quad}0<s\le 1,
\end{equation}
there exists $C$ that depends on $\lam,\Lam$ and $f$ such that
\begin{equation}\label{eq:expliciterr1}
\begin{aligned}
\nm{\na(\cu-v_h)}{L^2(D)}&\le C\Lr{h^s+\abs{K}^{s/n}\eta(K)+e(\hmm)},\\
\nm{\cu-v_h}{L^2(D)}&\le C\Lr{h^s+\abs{K}^{s/n}\eta(K)+e(\hmm)}\Lr{h^s+\abs{K}^{s/n}\eta(K)}.
\end{aligned}
\end{equation}

\item If $f$ is supported in $K$ and $f\in L^2(D)$, then we may replace $\abs{K}^{s/n}\eta(K)$  in~\eqref{eq:expliciterr1} by $\abs{K}^{1/n}$.
\end{enumerate}
\end{coro}

\begin{proof}
The first assertion is straightforward.

To prove the second assertion, we just need to replace {\em Meyers' estimate} by the estimate~\eqref{eq:basicest} and apply the standard interpolation estimate.

As to the third assertion, denote by $q$ the conjugate index of $p$, we have
\begin{align*}
\nm{f}{W^{-1,p}(D)}&\le\sup_{g\in W_0^{1,q}(D)}\dfrac{\nm{f}{L^2(K)}\nm{g}{L^2(K)}}{\nm{g}{W^{1,q}(D)}}\\
&\le \abs{K}^{1/2-1/r}\sup_{g\in W_0^{1,q}(D)}\dfrac{\nm{f}{L^2(K)}\nm{g}{L^r(K)}}{\nm{g}{W^{1,q}(D)}}\\
&\le C\abs{K}^{1/2-1/r}\nm{f}{L^2(D)},
\end{align*}
where we have used the Sobolev imbedding
$W^{1,q}(D)\hookrightarrow L^{r}(D)$ with
$1/r=1/q-1/n$. Substituting the above estimate
into~\eqref{eq:diserrhomo1}, we obtain the improved factor
$\abs{K}^{1/n}$ in the estimate~\eqref{eq:expliciterr1}
\end{proof}%
\subsection{Accuracy for retrieving the local microscopic information}
Parallel to the above results, we have the following energy error estimate for $\uu-v_h$.
Our proof also relies on the {\em Meyers' regularity} result~\cite {Meyers:1963}
for Problem~\eqref{eq:ell} in an essential way. We state Meyers' results as follows. There exists $p_1>2$ that depends on $D,\Lam$ and $\lam$, such that for all $p\le p_1$,
\begin{equation}\label{eq:meyers1}
\nm{\na\uu}{L^p(D)}\le C \nm{f}{W^{-1,p}(D)}
\end{equation}
with $C$ depends on $D,\Lam$ and $\lam$.
%
\begin{lemma}
Let $\uu$ and $v_h$ be the solutions of Problem~\eqref{eq:ell} and Problem~\eqref{eq:varaarlequinapp},
respectively. There holds, for any $2<p<p_1$,
\[
\nm{\na(\uu-v_h)}{L^2(D)}\le C\Lr{\inf_{\chi\in X_h}\nm{\na(\uu-\chi)}{L^2(D)}+\abs{D\setminus K}^{1/2-1/p}{\nm{f}{W^{-1,p}(D)}}},
\]
where $C$ depends on $\lam,\lam',\Lam,\Lam'$ and $D$.
\end{lemma}

The proof is omitted because it follows the same line that leads to~\eqref{eq:diserrhomo1} except that the Meyers' estimate~\eqref{eq:meyers1} for $\uu$ is exploited. 

The above estimate indicates that the global microscopic information
can be retrieved provided that $\abs{K}$ is big, namely we
solve~\eqref{eq:ell} almost everywhere, which seems to contradict with
our motivation. In fact, our projective is less ambitious, since what
we need is the local microscopic information. Therefore, the most
relevant error notion is often related to the local norm instead of
the global energy error. The following local energy estimate is the
main result of this part.

We assume that $\rho\equiv 1$ on
$K_0\subset\subset K$, and
$\text{dist}(K_0, \pa K\setminus\pa D)=d\ge \kappa h$ for a
sufficiently large $\kappa>0$, moreover $\abs{\na\rho}\le C/d$ for
certain constant
$C$. 
For a subset $B\in D$, we define
\[
H_{<}^1(B){:}=\set{u\in H^1(D)}{u|_{D\setminus B}=0}.
\]
In order to prove the localized energy error estimate, we state some properties of $X_h$ confined to $K$ following those
of~\cite{Demlow:2011}. More detailed discussion on these properties may be found in~\cite{NitscheSchatz:1974}. Let $G_1$ and $G$ be subsets of $K$ with $G_1\subset G$ and $\text{dist}(G_1,\pa G\setminus \pa D)=\wt{d}>0$. The following assumptions are assumed to hold:
\begin{enumerate}
\item[A1:]{\em Local interpolant.} There exists a local interpolant such that for any $u\in H_{<}^1(G_1)\cap C(G_1)$,
$I u\in X_h\cap H_{<}^1(G)$.

\item[A2:]{\em Inverse properties.} For each $\chi\in X_h$ and $\tau\in\mc{T}_h\cap K, 1\le p\le q\le\infty$, and
$0\le\nu\le s\le r$,
\begin{equation}\label{eq:inverse}
\nm{\chi}{W^{s,q}(\tau)}\le Ch_{\tau}^{\nu-s+n/p-n/q}\nm{\chi}{W^{\nu,p}(\tau)}.
\end{equation}

\item[A3.]{\em Superapproximation.} Let $\omega\in C^\infty(K)\cap H_{<}^1(G_1)$ with
$\abs{\omega}_{W^{j,\infty}(K)}\le C\wt{d}^{-j}$ for integers $0\le j\le r$ for each $\chi\in X_h\cap H_{<}^1(G)$
and for each $\tau\in\mc{T}_h\cap K$ satisfying $h_{\tau}\le d$,
\begin{equation}\label{eq:suppapp}
\nm{\omega^2\chi-I(\omega^2\chi)}{H^1(\tau)}\le C\Lr{\dfrac{h_{\tau}}{\wt{d}}\nm{\na(\omega\chi)}{L^2(\tau)}
+\dfrac{h_{\tau}}{\wt{d}^2}\nm{\chi}{L^2(\tau)}},
\end{equation}
where the interpolant $I$ is defined in A$1$.
\end{enumerate}

The assumptions A$1$,A$2$ and A$3$ are satisfied by standard Lagrange finite element defined on shape-regular
grids. In particular, the {\em  Superapproximation} property~\eqref{eq:suppapp} was recently proved
in~\cite[Theorem 2.1]{Demlow:2011}, which is the key for the validity of the local energy estimate over
shape-regular grids.
\begin{theorem}\label{thm:dislener}
Let $K_0\subset K\subset D$ be given, and let $\text{dist}(K_0,\pa K\setminus\pa D)=d$. Let assumptions A$1$, A$2$ and A$3$ hold with $\wt{d}=d/16$, in addition, let $\max_{\tau\cap K\not=\emptyset}h_{\tau}/d\le 1/16$. Then
\begin{equation}\label{eq:dislener}
\nm{\na(\uu-v_h)}{L^2(K_0)}\le C\Lr{\inf_{\chi\in X_h}\nm{\na(\uu-\chi)}{L^2(K)}+d^{-1}\nm{\uu-v_h}{L^2(K)}},
\end{equation}
where $C$ depends only on $\Lam,\lam,\Lam',\lam'$ and $D$.
\end{theorem}

\begin{remark}
The estimate~\eqref{eq:dislener} is also valid even if the subdomain $K_0$ abuts the original
domain $D$, which makes practical implementation convenient.
\end{remark}

In the above local energy
estimate, the first contribution comes from local approximation, the local events may be resolved by the adaptive method that may require highly
refined mesh, which is allowed by the above theorem. All the other
contributions are encapsulated in the second term $\nm{\uu-v_h}{L^2(K)}$, which is a direct consequence of the $L^2$ estimate~\eqref{eq:diserrhomo2} and the triangle inequality as follows. To make the presentation simpler, we assume the regularity estimate~\eqref{eq:ellreg} is valid with $s=1$, then
\begin{equation}\label{eq:disl2}
\begin{aligned}
\nm{\uu-v_h}{L^2(K)}&\le\nm{\uu-v_h}{L^2(D)}\le\nm{\cu-v_h}{L^2(D)}+\nm{\uu-\cu}{L^2(D)}\\
&\le C\Lr{h^2+\abs{K}^{2/n}\eta^2(K)+\max_{x\in D\setminus K}\|(\A-\A_h)(x)\|(h+\abs{K}^{1/n}\eta(K))}\nn\\
&\quad+\nm{\uu-\cu}{L^2(D)}.
\end{aligned}
\end{equation}

The convergence rate of the approximated solution in $L^2$ consists of
two parts, the first one is how the solution approximates the
homogenized solution, which relies on the smoothness of the
homogenized solution, the size of the support of the transition
function $\rho$, and the error committed by the approximation of the
effective matrix. The second source of the error comes from the
convergence rate in $L^2$ for the homoginization problem. For any
bounded and measurable $\a$, $\nm{\uu-\cu}{L^2(D)}$ converges to zero
as $\eps$ tends to zero by H-convergence theory. More structures have
to be assumed on $\a$ if one were to seek for a convergence rate. There
are a lot of results for estimating $\nm{\uu-\cu}{L^2(D)}$ under
various conditions on $\a$. Roughly speaking,
$\nm{\uu-\cu}{L^2(D)}\simeq\mc{O}(\eps^{\gamma})$, where $\gamma$
depends on the properties of the coefficient $\a$ and the domain
$D$. We refer to~\cite{KenigLinShen:2012} for a careful study of this
problem for elliptic system with periodic coefficients.  For elliptic
systems with almost-periodic coefficients, we refer
to~\cite{Shen:2015} and references therein for related
discussions.

The proof of this theorem is in the same spirit of~\cite{NitscheSchatz:1974} by combining the ideas of~\cite{Demlow:2011} and~\cite{Schatz:2005}. In particular, the following Caccioppoli-type estimate for {\em discrete harmonic function}
is a natural adaption of~\cite[Lemma 3.3]{Demlow:2011}, which is crucial for the local energy error estimate.
\begin{lemma}
Let $K_0\subset K\subset D$ be given, and let $\text{dist}(K_0,\pa K\setminus\pa D)=d$. Let assumptions A$1$, A$2$ and A$3$ hold with $\wt{d}=d/4$, and assume that $u_h\in X_h$ satisfies
\[
\inner{\bb_h\na u_h}{\na v}=0\qquad\forall v\in X_h\cap H_{<}^1(K).
\]
In addition, let $\max_{\tau\cap K\not=\emptyset}h_{\tau}/d\le 1/4$. Then, there exists $C$ such that
\begin{equation}\label{eq:disharmonic}
\nm{u_h}{H^1(K_0)}\le\dfrac{C}{d}\nm{u_h}{L^2(K)},
\end{equation}
where $C$ depends only on the constants in~\eqref{eq:inverse},~\eqref{eq:suppapp}, $\Lam'$ and $\lam'$.
\end{lemma}

\smallskip
\begin{proof}[Proof of Theorem~\ref{thm:dislener}]
Without loss of generality, we may assume that $K_0$ is the intersection of a ball $B_{d/2}$ with $D$, the general case
may be proved by a covering argument as in~\cite[Theorem 5.1 and Theorem 5.2]{NitscheSchatz:1974}. Let
$\wt{K}$ be the intersection of a ball $B_{3d/4}$ with $D$ and $K$ be the intersection of a ball $B_d$ with $D$. Therefore, we have $K_0\subset\subset\wt{K}\subset\subset K$, and $\text{dist}(K_0,\pa\wt{K}\setminus\pa D)=d/4$. Let $\wh{u}=\rho\uu$ and define
$\wh{u}_h\in X_h\cap H_{<}^1(K)$ as the local Galerkin projection of $\wh{u}$, i.e., $\wh{u}_h\in X_h\cap H_{<}^1(K)$
satisfying
\[
\inner{\bb_h\na(\wh{u}-\wh{u}_h)}{\na v}=F(v)\qquad\forall v\in X_h\cap H_{<}^1(K),
\]
where $F(v){:}=\inner{(\bb_h-\a)\na\uu}{\na v}$.
By coercivity of $\bb_h$, we immediately have the stability estimate
\begin{equation}\label{eq:stability}
\nm{\na\wh{u}_h}{L^2(K)}\le C\Lr{\nm{\na\wh{u}}{L^2(K)}+\nm{\na\uu}{L^2(K\cap K_1)}}
\end{equation}
for certain $C$ that depends only on $\Lam,\lam,\Lam'$ and $\lam'$.

By definition and recalling that $\rho\equiv 1$ on $\wt{K}$, we may verify that $\wh{u}_h-v_h$ is {\em discrete harmonic} in the sense that for any $v\in X_h\cap H_{<}^1(\wt{K})$, there holds
\begin{align*}
\inner{\bb_h\na(\wh{u}_h-v_h)}{\na v}&=\inner{\bb_h\na\wh{u}_h}{\na v}-\inner{f}{v}\\
&=\inner{\bb_h\na\wh{u}}{\na v}-F(v)-\inner{f}{v}\\
&=\inner{\bb_h\na\uu}{\na v}-\inner{\a\na\uu}{\na v}-F(v)\\
&=0.
\end{align*}

Using~\eqref{eq:disharmonic} and recalling that $\rho\equiv 1$ on $\wt{K}$, we obtain
\begin{align*}
\nm{\na(\uu-v_h)}{L^2(K_0)}&\le\nm{\na(\wh{u}-\wh{u}_h)}{L^2(K_0)}+\nm{\na(\wh{u}_h-v_h)}{L^2(K_0)}\\
&\le\nm{\na(\wh{u}-\wh{u}_h)}{L^2(K)}+\dfrac{C}{d}\nm{\wh{u}_h-v_h}{L^2(\wt{K})}\\
&\le\nm{\na(\wh{u}-\wh{u}_h)}{H^1(K)}+\dfrac{C}{d}\Lr{\nm{\wh{u}_h-\wh{u}}{L^2(\wt{K})}+\nm{\uu-v_h}{L^2(\wt{K})}}.
\end{align*}
Using Poincare's inequality, we obtain
\[
\nm{\wh{u}_h-\wh{u}}{L^2(\wt{K})}\le Cd\nm{\na(\wh{u}_h-\wh{u})}{L^2(\wt{K})}.
\]
Combining the above two inequalities, we obtain
\begin{equation}\label{eq:med2}
\nm{\na(\uu-v_h)}{L^2(K_0)}\le C\nm{\na(\wh{u}-\wh{u}_h)}{L^2(K)}+\dfrac{C}{d}\nm{\uu-v_h}{L^2(\wt{K})}.
\end{equation}

Next we use the triangle inequality and~\eqref{eq:stability}, we obtain, 
\begin{align*}
\nm{\na(\wh{u}-\wh{u}_h)}{L^2(K)}&\le C\Lr{\nm{\na\wh{u}}{L^2(K)}+\nm{\na\uu}{L^2(K\cap K_1)}}\\
&\le C\Lr{\nm{\na\uu}{L^2(K)}+d^{-1}\nm{\uu}{L^2(K)}+\nm{\na\uu}{L^2(K\cap K_1)}}\\
&\le C\Lr{\nm{\na\uu}{L^2(K)}+d^{-1}\nm{\uu}{L^2(K)}}.
\end{align*}
Substituting the above inequality into~\eqref{eq:med2}, we obtain
\[
\nm{\na(\uu-v_h)}{L^2(K_0)}\le C\Lr{
\nm{\na\uu}{L^2(K)}+d^{-1}\nm{\uu}{L^2(K)}+\dfrac{C}{d}\nm{\uu-v_h}{L^2(\wt{K})}}.
\]
For any $\chi\in X_h$, we write $\uu-v_h=(\uu-\chi)-(v_h-\chi)$, and we employ the above inequality with $\uu$ taken to be $\uu-\chi$ and
$v_h$ taken to be $v_h-\chi$. This implies
\begin{equation}\label{eq:medlocal}
\begin{aligned}
\nm{\na(\uu-v_h)}{L^2(K_0)}&\le C\inf_{\chi\in X_h}\Lr{
\nm{\na(\uu-\chi)}{L^2(K)}+d^{-1}\nm{\uu-\chi}{L^2(K)}}\\
&\quad+\dfrac{C}{d}\nm{\uu-v_h}{L^2(\wt{K})}.
\end{aligned}
\end{equation}
Let $\chi^\ast=\arg\inf_{\chi\in X_h}\nm{\na(\uu-\chi)}{L^2(K)}$, we let
\[
\chi=\chi^\ast+\negint_K(\uu-\chi^\ast)\dx.
\]
Using Poincare's inequality, there exists $C$ independent the size of $K$ such that
\[
\nm{\uu-\chi}{L^2(K)}=\nm{\uu-\chi^\ast-\negint_K(\uu-\chi^\ast)}{L^2(K)}\le Cd\nm{\na(\uu-\chi^\ast)}{L^2(K)}.
\]
Substituting the above inequality into~\eqref{eq:medlocal}, we
obtain~\eqref{eq:dislener} and complete the proof.
\end{proof}
\section{Numerical Examples}\label{sec:numer}
In this section, we present some numerical examples to demonstrate the accuracy and efficiency of the proposed method. The domain $K_0$ whose microstructure
is of interest is assumed to be a square for simplicity, i.e., $K_0=(-L,L)^2$. The first step in implementation is to construct the transition function $\rho$.
For any parameter $\delta>0$, we let $\gamma_\delta: [-L-\delta,L+\delta]\to [0,1]$ be a first order differentiable function
such that $\gamma_\delta(t)\equiv 1$ for $0\le t\le L$, and $\gamma_\delta^\prime(L)=0,\gamma_\delta(L+\delta)=\gamma_\delta^\prime(L+\delta)=0$. Moreover, $\gamma_\delta$
is an even function with respect to the origin. We may extend $\gamma_\delta$ to a function defined on $\RR$ by taking
$\gamma_\delta(t)\equiv 0$ for $\abs{t}\ge L+\delta$. Finally we define the transition
function $\rho(x){:}=\gamma_\delta(x_1)\gamma_\delta(x_2)$. In particular, we denote $K{:}=(-L-\delta,L+\delta)^2$.

For both examples, we compute
\[
\nm{u^\eps-v_h}{H^1(K_0)}\qquad\text{and}\qquad\nm{u_0-v_h}{H^1(D\backslash K)},
\]
which are the two quantities of major interest, and the domain $D=(0,1)^2$. The reference solutions $u^\eps$ and $u_0$ are not available analytically, we compute both $u^\eps$ and $u_0$ over very refined mesh, and the details will be given below.
\subsection{An example with two scales}\label{subsec:2ex}
In this example, we take
\[
a^\eps(x)=\dfrac{(R_1+R_2\sin(2\pi x_1))(R_1+R_2\cos(2\pi x_2))}{(R_1+R_2\sin(2\pi x_1/\eps))(R_1+R_2\sin(2\pi x_2/\eps))}I,
\]
where $I$ is a two by two identity matrix. The effective matrix is given by
\[
\mc{A}(x) =\dfrac{(R_1+R_2\sin(2\pi x_1))(R_1+R_2\cos(2\pi x_2))}{R_1\sqrt{R_1^2-R_2^2}}I.
\]
In the simulation, we let $R_1 = 2.5, R_2 = 1.5$ and $\eps=0.01$. The forcing term $f\equiv 1$ and the
homogeneous Dirichlet boundary condition is imposed. The subdomain around the defect is
$K_0 =(0.5,0.5)+(-L,L)^2$ with $L=0.05$.

We compute $u^\eps$ with $P_1$ Lagrange element over a uniform mesh with mesh size $3.33e-4$, which amounts to putting
$30$ points inside each wave length, i.e., $h\simeq\eps/30$. The
homogenized solution $u_0$ is computed by directly solving the
homogenized problem~\eqref{homoell} with $P_1$ Lagrange
element over a uniform mesh with mesh size $3.33e-4$, and the above
analytical expression for $\mc{A}$ is employed in the simulation. We take these
numerical solutions as the reference solutions, which are still
denoted by $u^\eps$ and $u_0$, respectively. Note that the mesh is a
body-fitted grid with respective to the defect domain $K_0$.

We solve the problem~\eqref{eq:varaarlequinapp} over a non-uniform body-fitted mesh as in Figure~\ref{Abdullecoefununiformmesh}.
\begin{figure}[h]
  \begin{minipage}{0.48\linewidth}
    \centerline{\includegraphics[width=6cm]{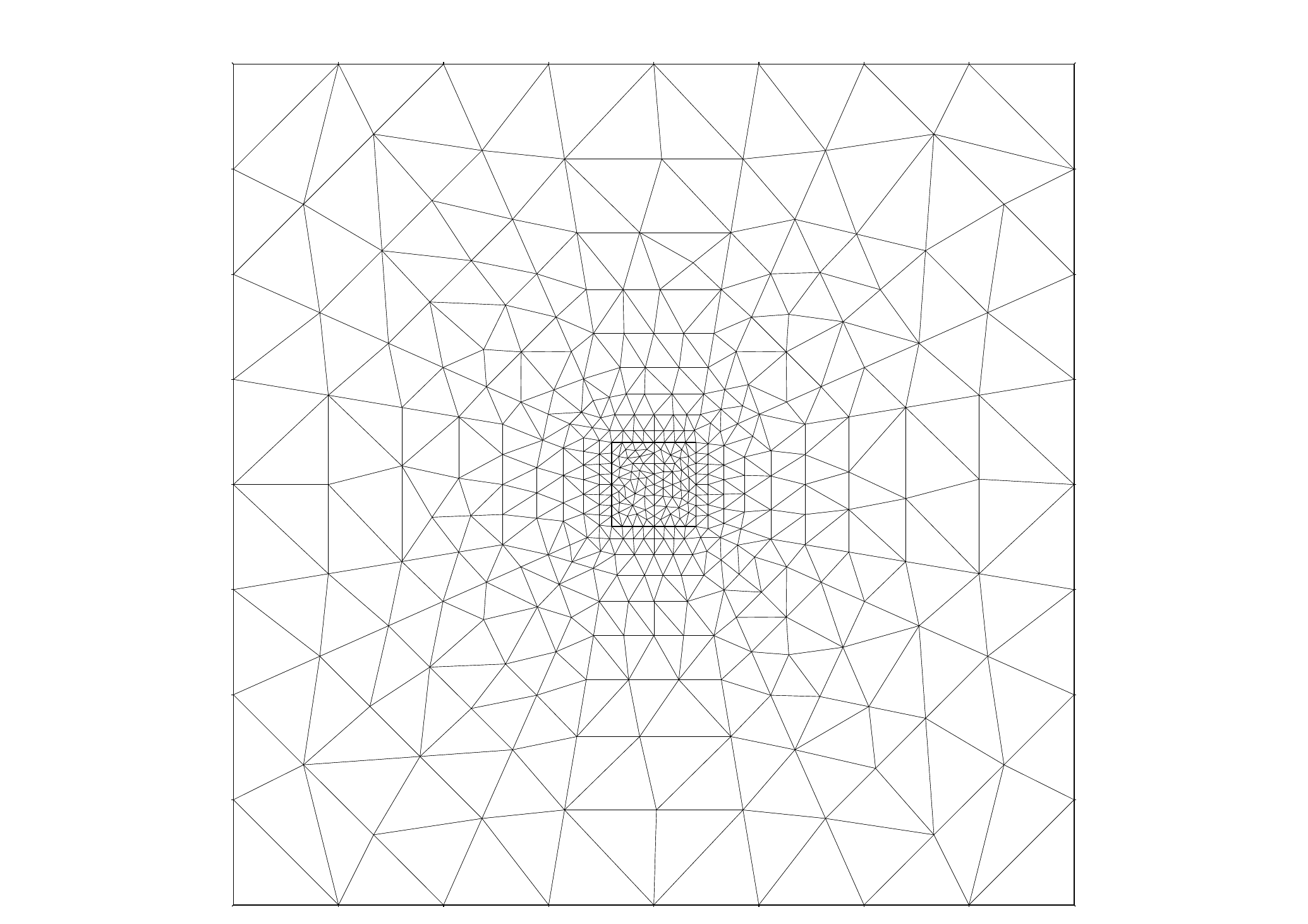}}
  \end{minipage}
  \caption{Mesh for the concurrent method}\label{Abdullecoefununiformmesh}
\end{figure}
The solutions $\uu,\cu$ and $v_h$ are plotted in
Figure~\ref{tangcouplSoluFigview1}, which look almost identical to
each other on the macroscopic scale.
\begin{figure}[h]
		\begin{minipage}{0.48\linewidth}
			\centerline{\includegraphics[width=16cm,height=4cm]{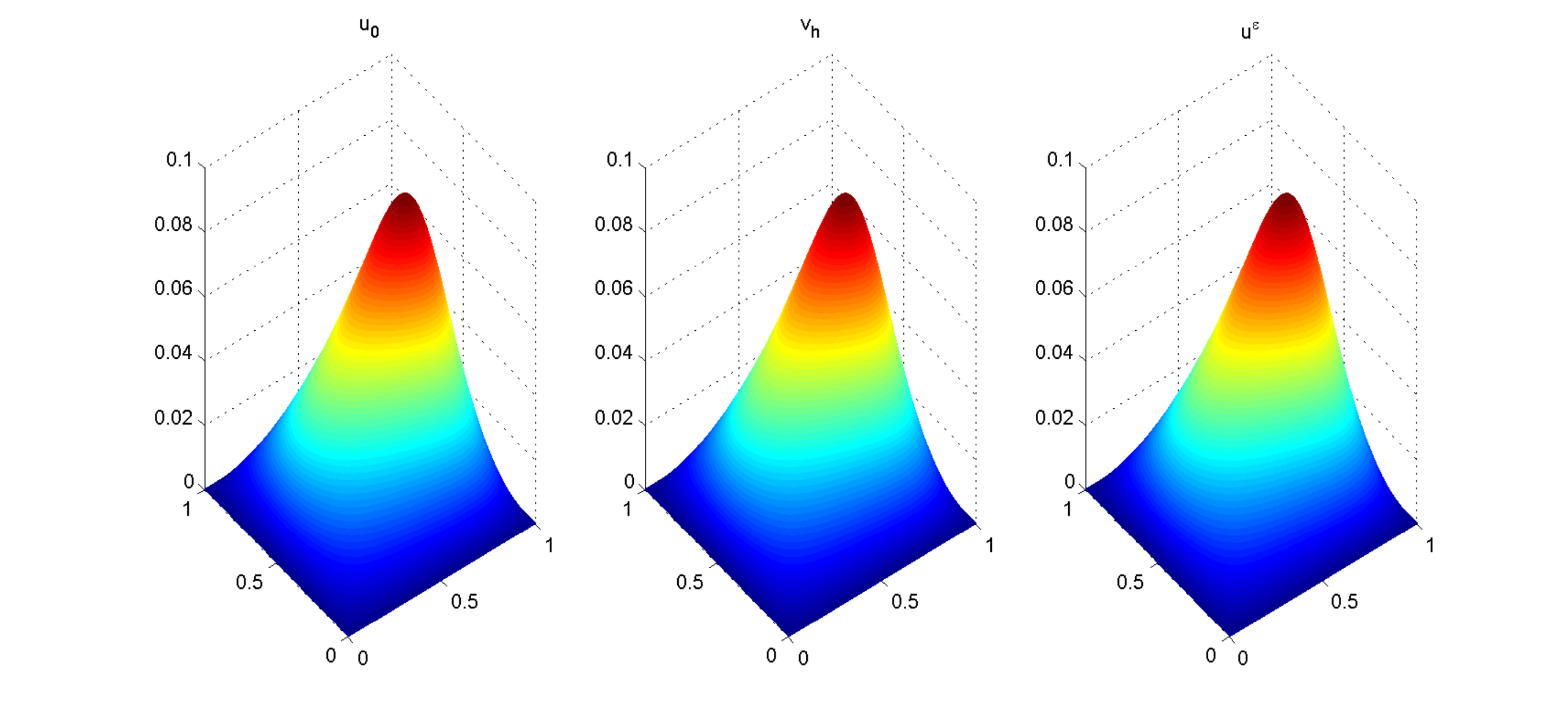}}
		\end{minipage}\hfill\vfill
\caption{Solutions in D and $\delta=0.05$. Left: the homogenized solution $\cu$; Middle: the solution of the concurrent method; Right: the solution of Problem~\eqref{eq:ell}.}\label{tangcouplSoluFigview1}
\end{figure}

We plot the zoomed-in solutions inside the defect domain $K_0$ in
Figure~\ref{view2}, it seems the hybrid solution approximates the
original solution very well because it captures the oscillation of the
microstructures inside $K_0$.
\begin{figure}[h]
	\begin{minipage}{0.48\linewidth}
		\centerline{\includegraphics[width=16cm,height=3cm]{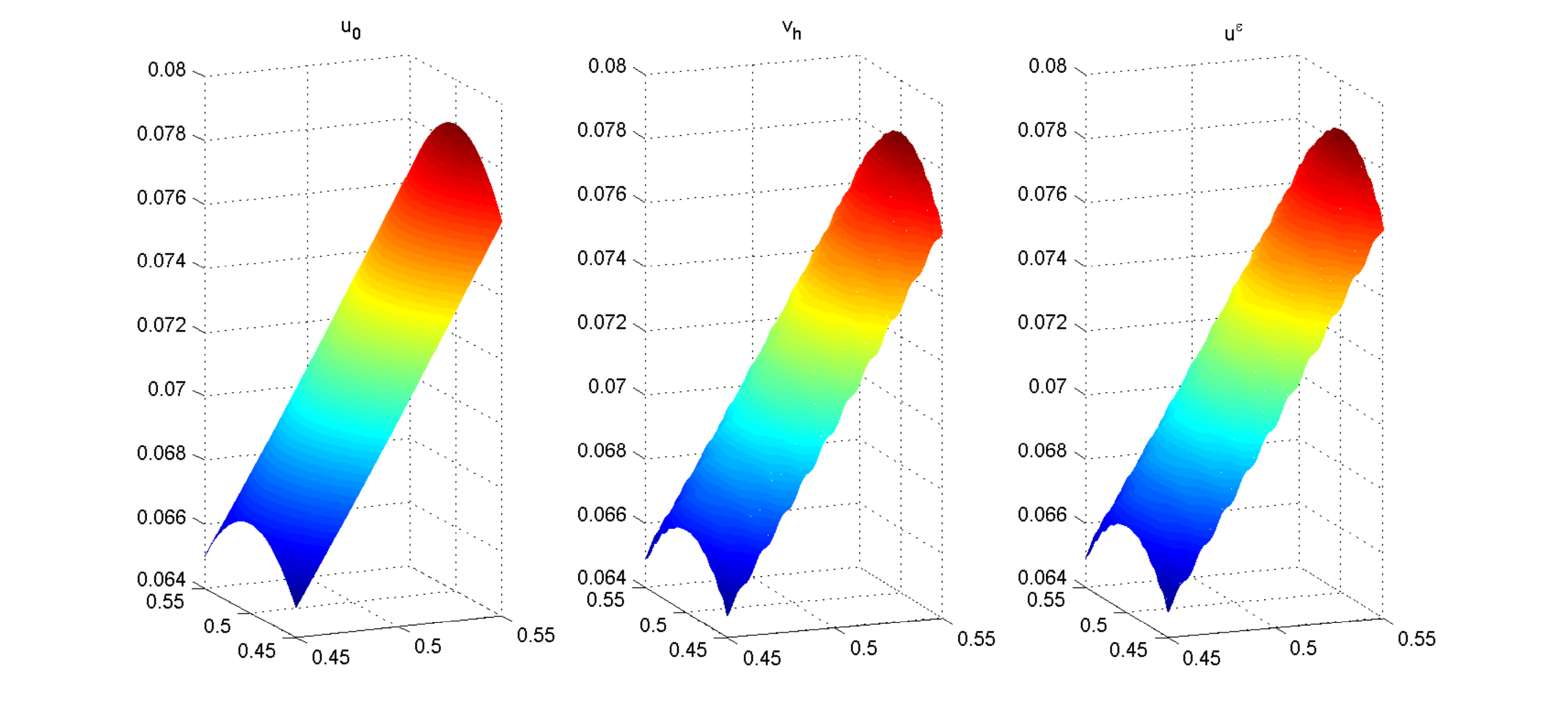}}
	\end{minipage}
	\hfill\vfill\caption{Solutions inside the defect domain $K_0$ and $\delta=0.05$. Left: the homogenized solution $\cu$; Middle: the solution of the concurrent method;
Right: the solution of Problem~\eqref{eq:ell}.}\label{view2}
\end{figure}

Next we plot the zoomed-in solutions in $D\backslash K$ in Figure~\ref{tangcouplSoluFigview3}, i.e., outsie the defect domain $K_0$, it seems that the hybrid solution approximates
the homogenized solution very well because it is as smooth as the homogenized solution, while there is oscillation in $\uu$.
\begin{figure}[h]
	\begin{minipage}{0.48\linewidth}
		\centerline{\includegraphics[width=16cm,height=3cm]{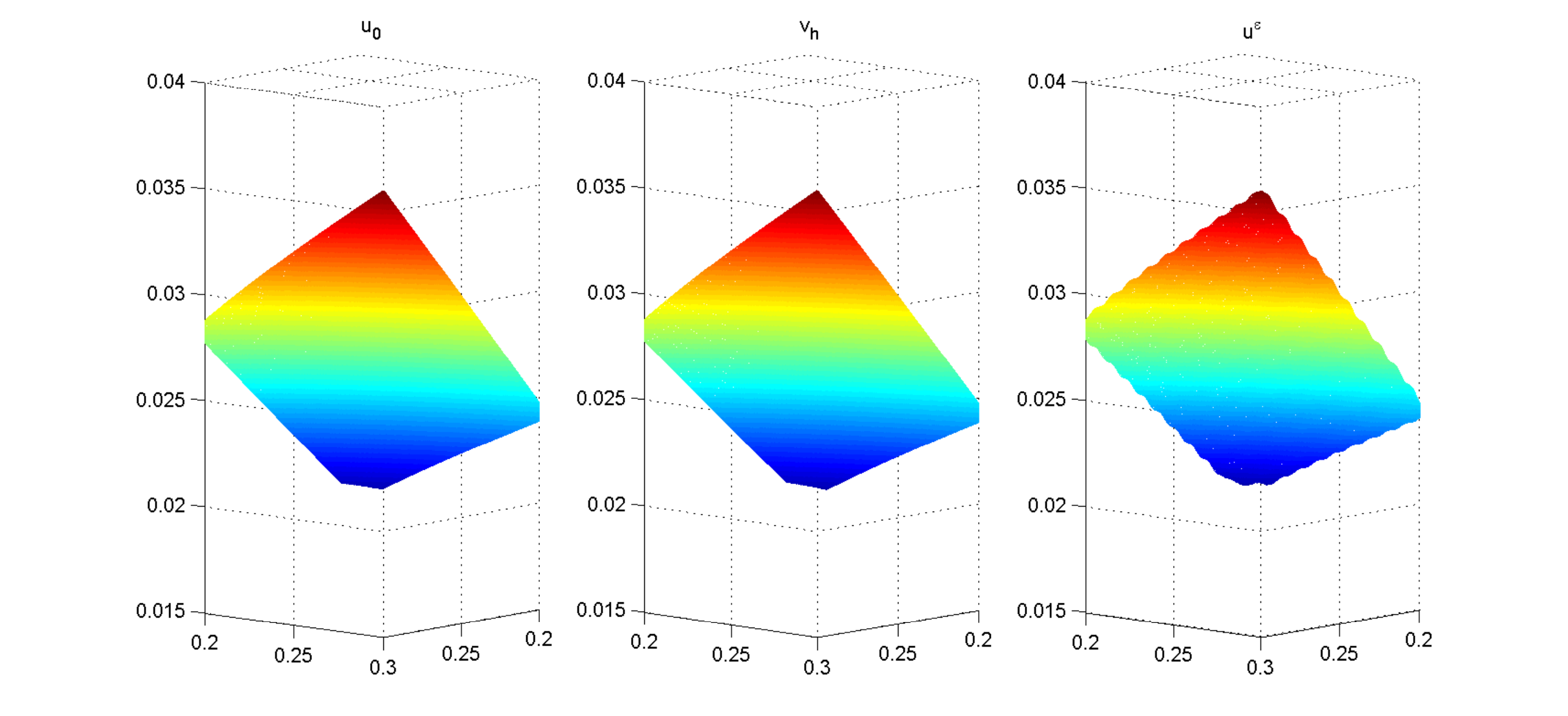}}
	\end{minipage}
\hfill\vfill
	\caption{Solution in a subdomain of $D/K$ and $\delta=0.05$.  Left: the homogenized solution $\cu$; Middle: the solution of the concurrent method; Right: the solution of Problem~\eqref{eq:ell}.}\label{tangcouplSoluFigview3}
\end{figure}

The localized error $\nm{u^\eps-v_h}{H^1(K_0)}$ is shown in Figure~\ref{Ex1Fig1}(a) for both the smooth and nonsmooth transition functions. Here the nonsmooth transition function is the characteristic function of $K_0$. The results show first that the parameter $\delta$ has little influence on the local energy error, and secondly the result obtained by the nonsmooth transition function is less accurate.
\begin{figure}[htbp]
		\begin{minipage}{0.48\linewidth}
	\centerline{\includegraphics[width=6.8cm]{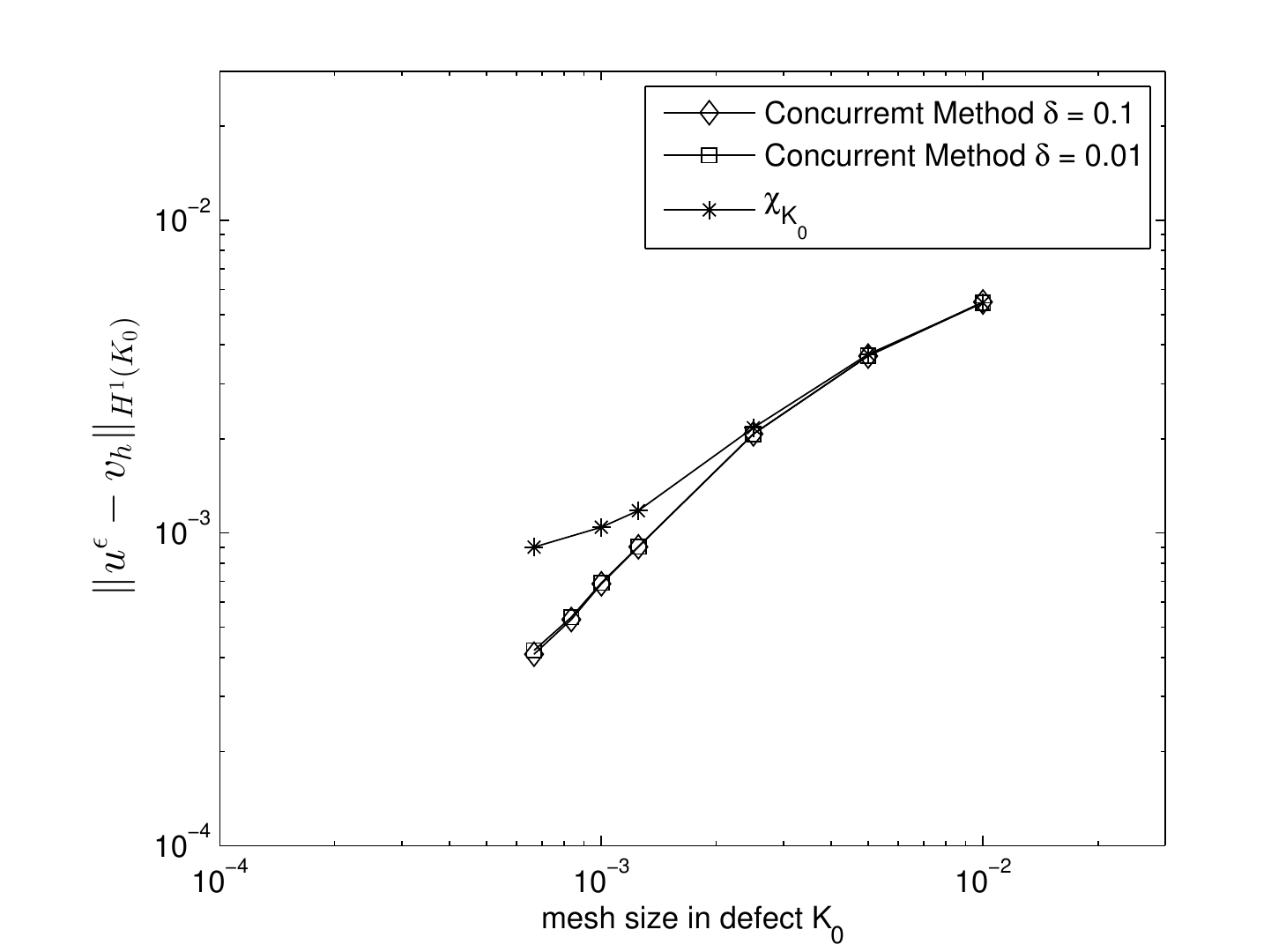}}
	\centerline{(a)}
 	\end{minipage}
	\begin{minipage}{0.48\linewidth}
		\centerline{\includegraphics[width=6.8cm]{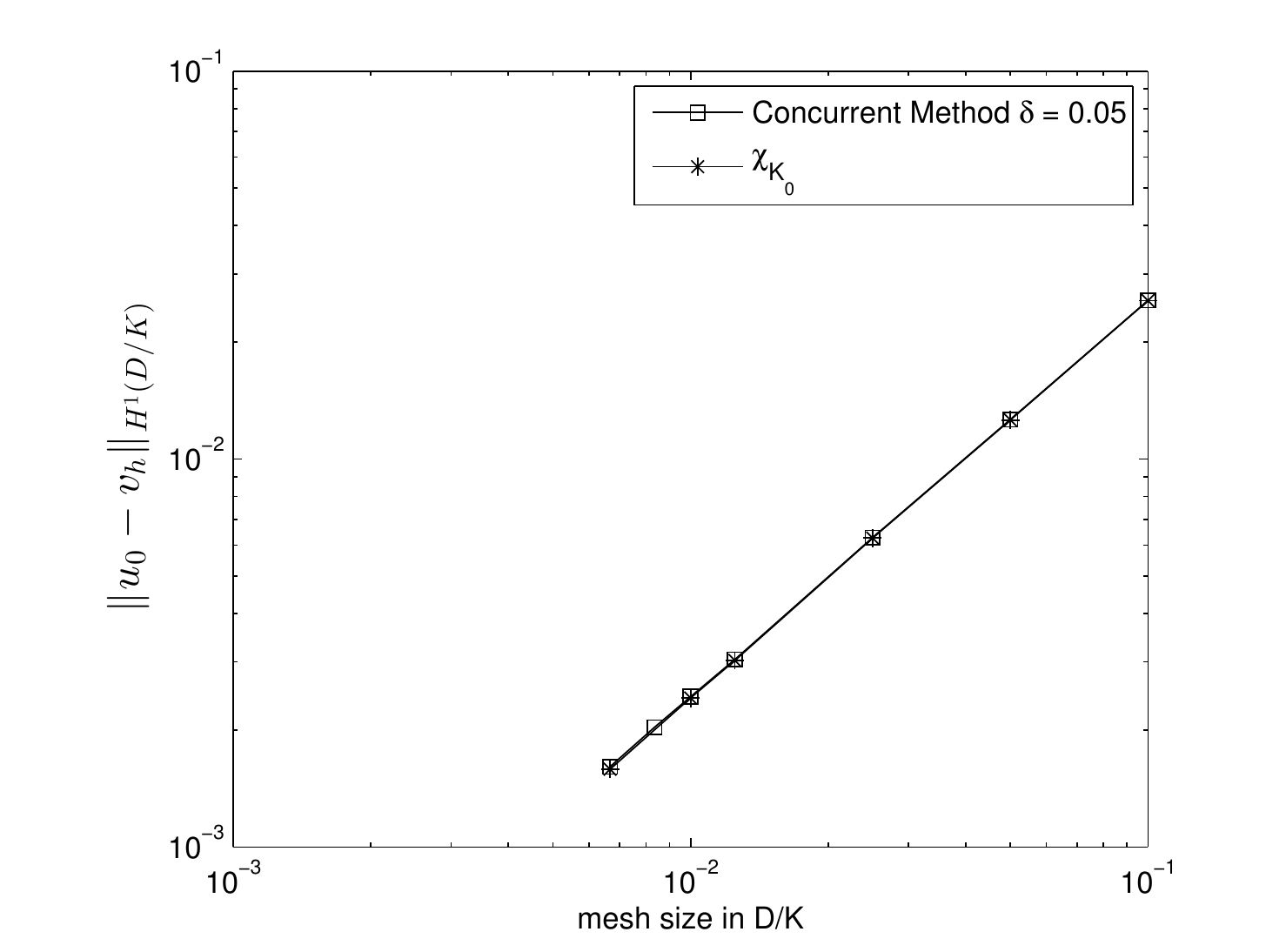}}
			\centerline{(b)}
	\end{minipage}
	\caption{(a) Localized H$^1$ error with respective to the mesh size inside the defect domain $K_0$. (b) Error between the hybrid solution and the homogenized solution outside $K$.}
	\label{Ex1Fig1}
\end{figure}

Finally, we plot the error $\nm{u_0-v_h}{H^1(D\backslash K)}$ in Figure~\ref{Ex1Fig1}(b) for both the smooth and nonsmooth transition functions. For fixed $L$, this quantity decreases as the mesh outside $K$ is refined. It seems that the smoothness of the transition function has little effect on the accuracy of the homogenized solution, which is consistent with the theoretical results.

The results in Table~\ref{tab:homohybrid} show the convergence rate
of the hybrid solution to the homogenized solution. It is optimal in the sense that
the solution of the hybrid problem converges to the homogenized solution with first order in the energy norm, and it converges with second order in the $L^2$ norm. This seems
consistent with the theoretical estimates~\eqref{eq:diserrhomo1} and~\eqref{eq:diserrhomo2}, because
\[
\nm{\na(\cu-v_h)}{L^2(D)}\le C\Lr{h+L\abs{\ln L}^{1/2}},\quad
\nm{\cu-v_h}{L^2(D)}\le C\Lr{h^2+L^2\abs{\ln L}}.
\]
When $L\simeq h$, the convergence rate is of first order with respect to the energy norm, while the mesh size is smaller than $L$, the dominant term in the error bound is $L\abs{\ln L}^{1/2}$, the convergence rate deteriorates a little bit, which is clear from the last line of Table~\ref{tab:homohybrid}. The same scenario applies to the $L^2$ error estimate.
\begin{table}[htbp]
  \centering
  \caption{\small Error between the hybrid solution and the homogenized solution outside $K$.}\label{tab:homohybrid}
  \begin{tabular}{|c|c|c|c|c|}\hline
    h   &$\nm{u_0-v_h}{L^2(D\backslash K)}$ &order  &$\nm{u_0-v_h}{H^1(D\backslash K)}$ &order  \\\hline
      1/10	&	9.04E-04	&		&	2.56E-02	&		\\ \hline
      1/20	&	2.47E-04	&	1.87	&	1.26E-02	&	1.02	\\ \hline
      1/40	&	7.96E-05	&	1.64	&	6.28E-03	&	1.01	\\ \hline
      1/80	&	3.41E-05	&	1.22	&	3.04E-03	&	1.04	\\ \hline
      1/160     	&	2.47E-05	&	0.47	&	1.61E-03	&	0.92	\\ \hline

  \end{tabular}
\end{table}
\subsection{An example without scale separation in the defect domain}\label{subsec:1ex}
The setup for the second example is the same with the first one except that the coefficient is replaced by
\(a^\eps =\chi_{K_0} \tilde{a}+(1-\chi_{K_0})\tilde{a}^\eps\),
where
\[
\tilde{a}(x)= 3 + \frac17\sum_{j=0}^4\sum_{i=0}^j\frac{1}{j+1}\cos\left( \left\lfloor 8(ix_2-\frac{x_1}{i+1})\right\rfloor+\left\lfloor 150ix_1\right\rfloor +\left\lfloor 150x_2\right\rfloor\right),
\]
and
\[
\tilde{a}^\eps(x)=\left(2.1+\cos(2\pi x_1/\eps)\cos(2\pi x_2/\eps)+\sin(4x_1^2x_2^2)\right)I.
\]
The above coefficient is taken from~\cite{Abdulle:15}, which has no clear scale inside $K_0$; while it is locally periodic outside $K_0$. We plot
the coefficient $a^\eps$ in Figure~\ref{AbdullecoefD0607} with $\eps = 0.1$. 
 \begin{figure}[h]
 	\begin{minipage}{0.48\linewidth}
 		\centerline{\includegraphics[width=6cm]{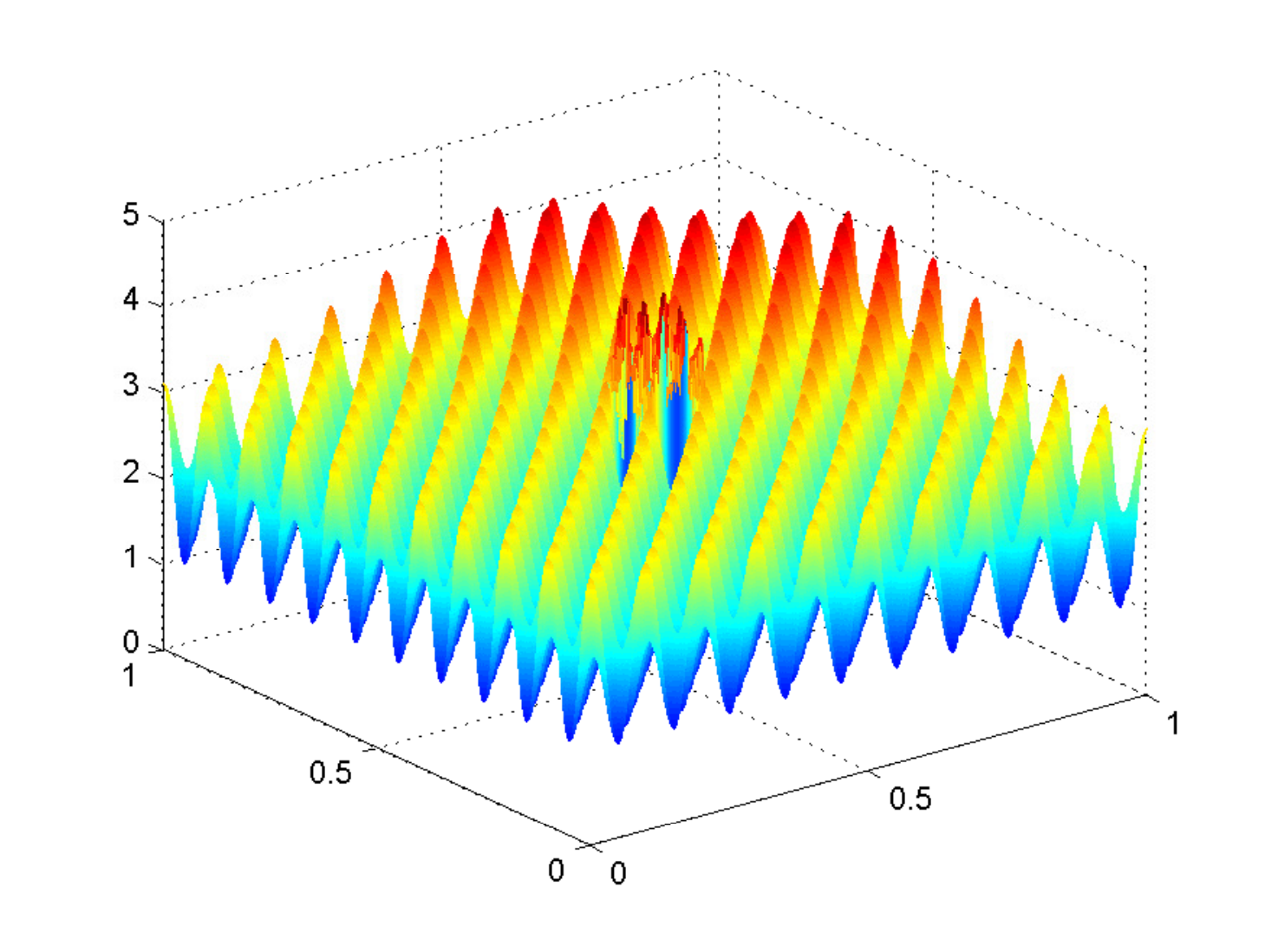}}
 	\end{minipage}
 	\hfill
 	\begin{minipage}{0.48\linewidth}
 		\centerline{\includegraphics[width=6cm]{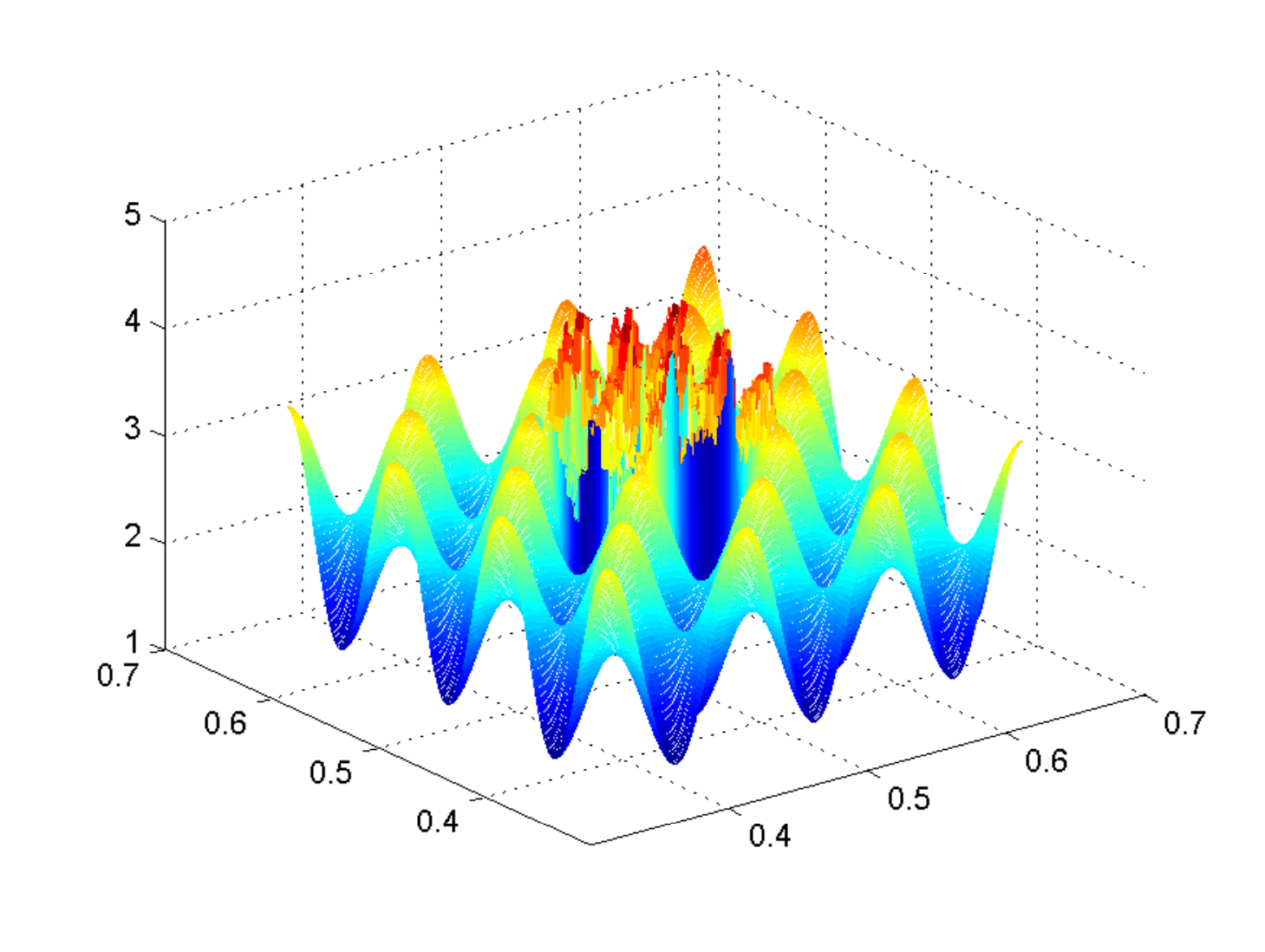}}
 	\end{minipage}
 	\caption{Coefficient $a^\eps$ with $\eps=0.1$. The right one is the zoomed-in plot near the defect.}
 	\label{AbdullecoefD0607}
 \end{figure}

We let $\eps=0.0063$ for the sake of comparison with those in~\cite{Abdulle:15} and compute $u^\eps$ over a uniform mesh with mesh size $3.33e-4$. By Corollary~\ref{coro:hlimit} and the identity~\eqref{eq:exact1}, the effective matrix $\mc{A}=\chi_{K_0}\wt{a}+(1-\chi_{K_0})\wt{A}$ and the approximating effective matrix $\mc{A}_h=\chi_{K_0}\wt{a}+(1-\chi_{K_0})\wt{A}_h$, where $\wt{A}_h$ is an approximation of the effective matrix associated with $\wt{a}^\eps$ through a fast solver based on the discrete least-squares reconstruction in the framework of HMM (see~\cite{LiMingTang:2012} and~\cite{HuangLiMing:2016} for details of such fast algorithm). We reconstruct $\wt{A}_h$ to high accuracy so that the reconstruction error is negligible. The homogenized solution $u_0$ is computed by solving Problem~\eqref{homoell} with $\mc{A}$, which has also been used in solving the hybrid problem~\eqref{eq:varaarlequinapp}, i.e.,
\begin{align*}
b_h^\eps&=\rho\chi_{K_0}\wt{a}+\rho(1-\chi_{K_0})\wt{a}^\eps+(1-\rho)\chi_{K_0}\wt{a}+(1-\rho)(1-\chi_{K_0})\wt{A}_h\\
&=\chi_{K_0}\wt{a}+(1-\chi_{K_0})\Lr{\rho\wt{a}^\eps+(1-\rho)\wt{A}_h}.
\end{align*}%

We solve Problem~\eqref{eq:varaarlequinapp} over a non-uniform mesh as in Figure~\ref{Abdullecoefununiformmesh}, and plot $\uu,v_h$ and $\cu$ and the zoomed-in solution in Figure~\ref{AbdullecoefD0607Abdullep0063solution}. The difference among them are small because there is no explicit scale inside the defect domain $K_0$.
\begin{figure}[h]
	\begin{minipage}{0.48\linewidth}
		\centerline{\includegraphics[width=16cm,height=4cm]{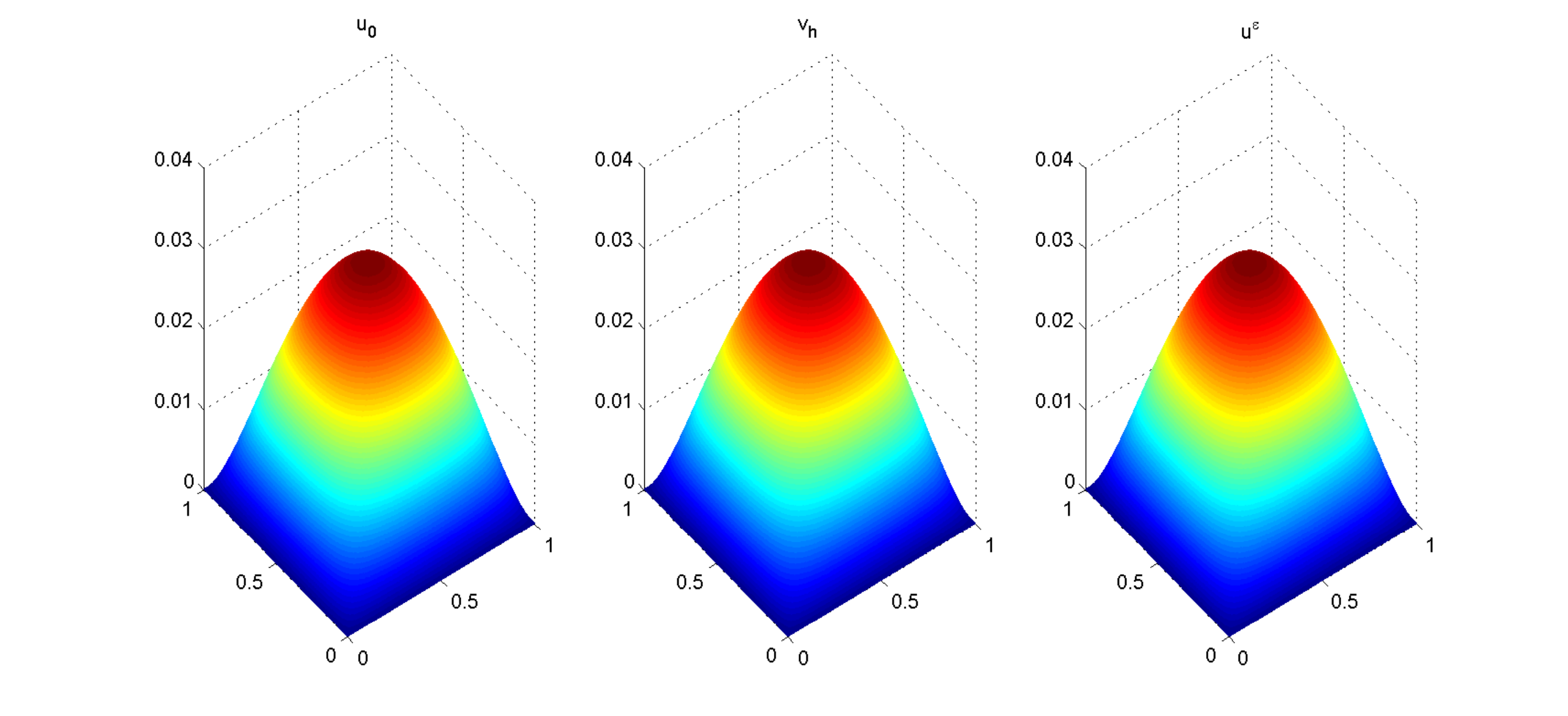}}
		\centerline{(a) Solutions in D}
	\end{minipage}
	\hfill\vfill
	\hfill\vfill
	\begin{minipage}{0.48\linewidth}
		\centerline{\includegraphics[width=16cm,height=3cm]{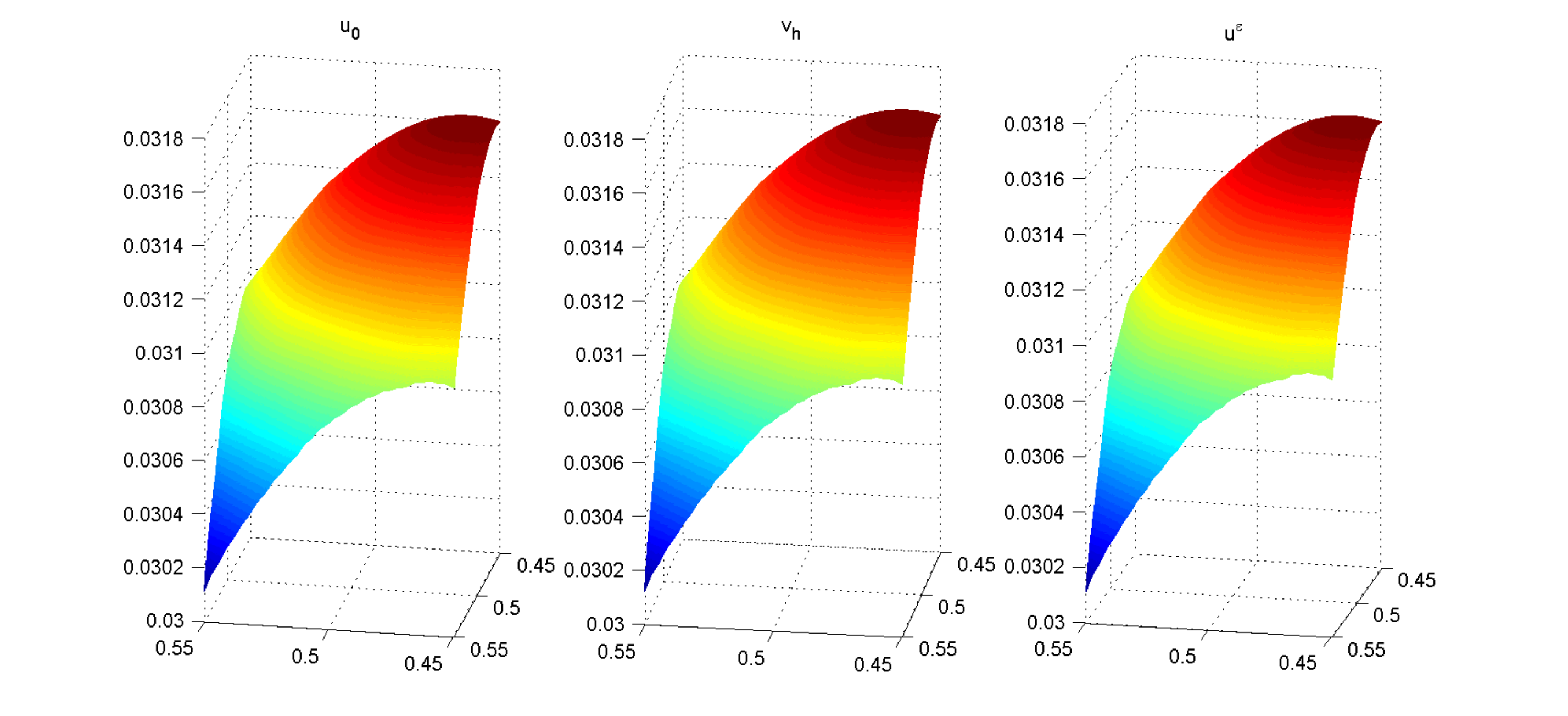}}
		\centerline{(b) Solutions in $K_0$.}
	\end{minipage}
			\hfill\vfill
			\begin{minipage}{0.48\linewidth}
				\centerline{\includegraphics[width=16cm,height=3cm]{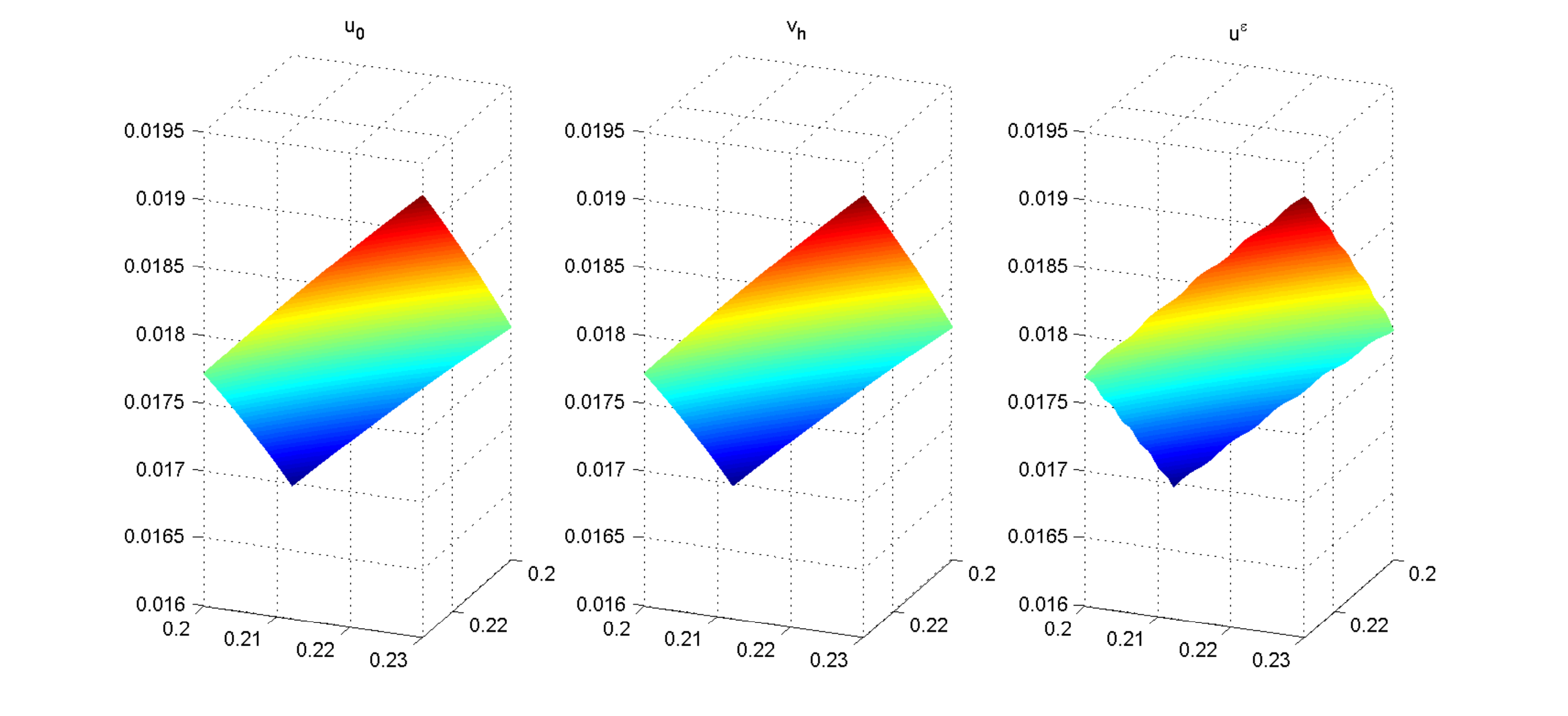}}
				\centerline{(c) Solutions in a subdomain of $D/K$.}
			\end{minipage}
	
	\caption{The solution of hybrid method with $\delta=0.05$ in the simulation.}\label{AbdullecoefD0607Abdullep0063solution}
\end{figure}

We plot the localized $H^1$ error in Figure~\ref{Ex2Fig1}(a). It seems the hybrid method converges slightly faster than the direct method, and the parameter $\delta$ has no significant effect on the results.
\begin{figure}[htbp]
	\begin{minipage}{0.48\linewidth}
		\centerline{\includegraphics[width=6.8cm]{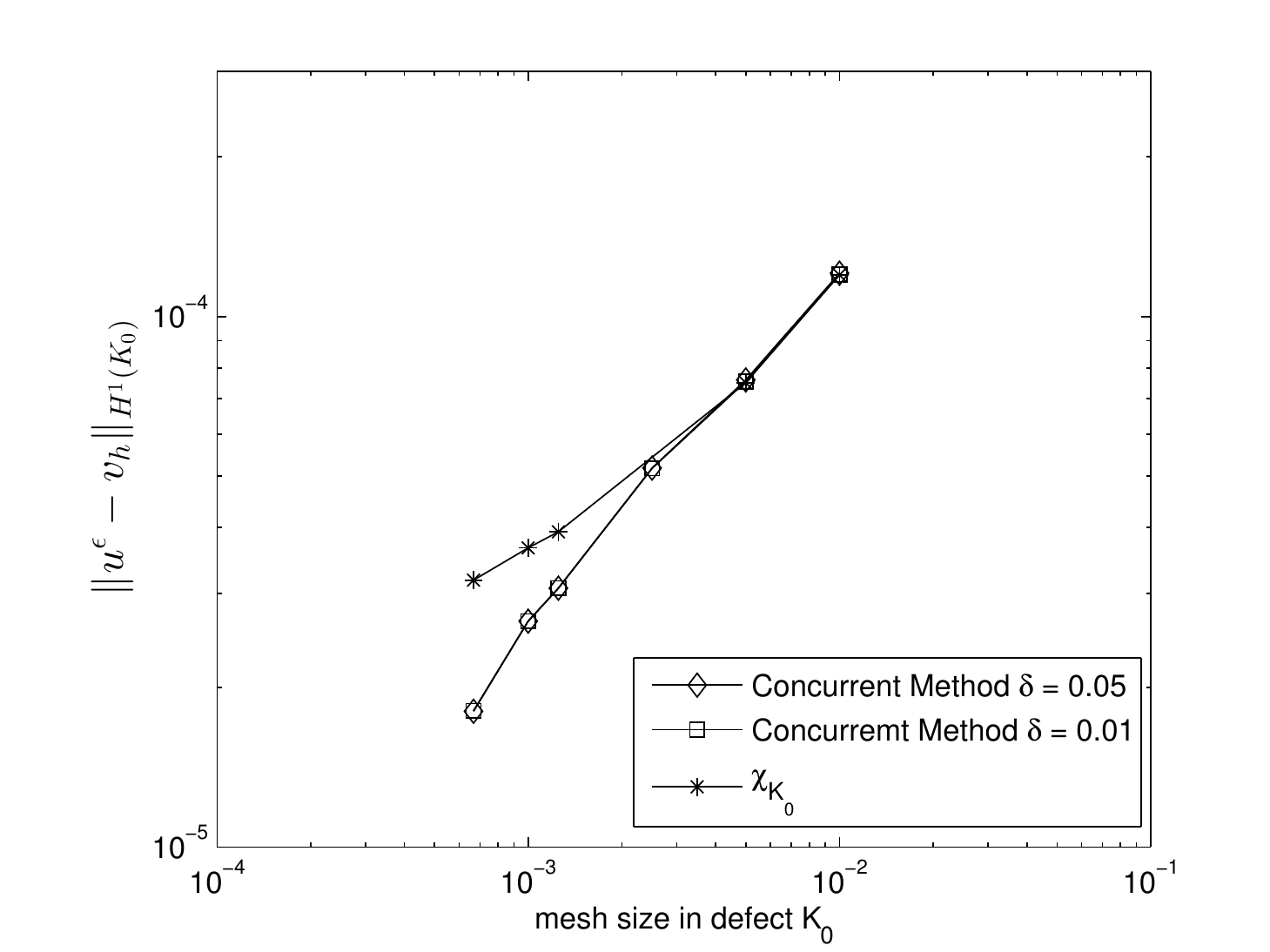}}
		\centerline{(a)}
	\end{minipage}
	\begin{minipage}{0.48\linewidth}
		\centerline{\includegraphics[width=6.8cm]{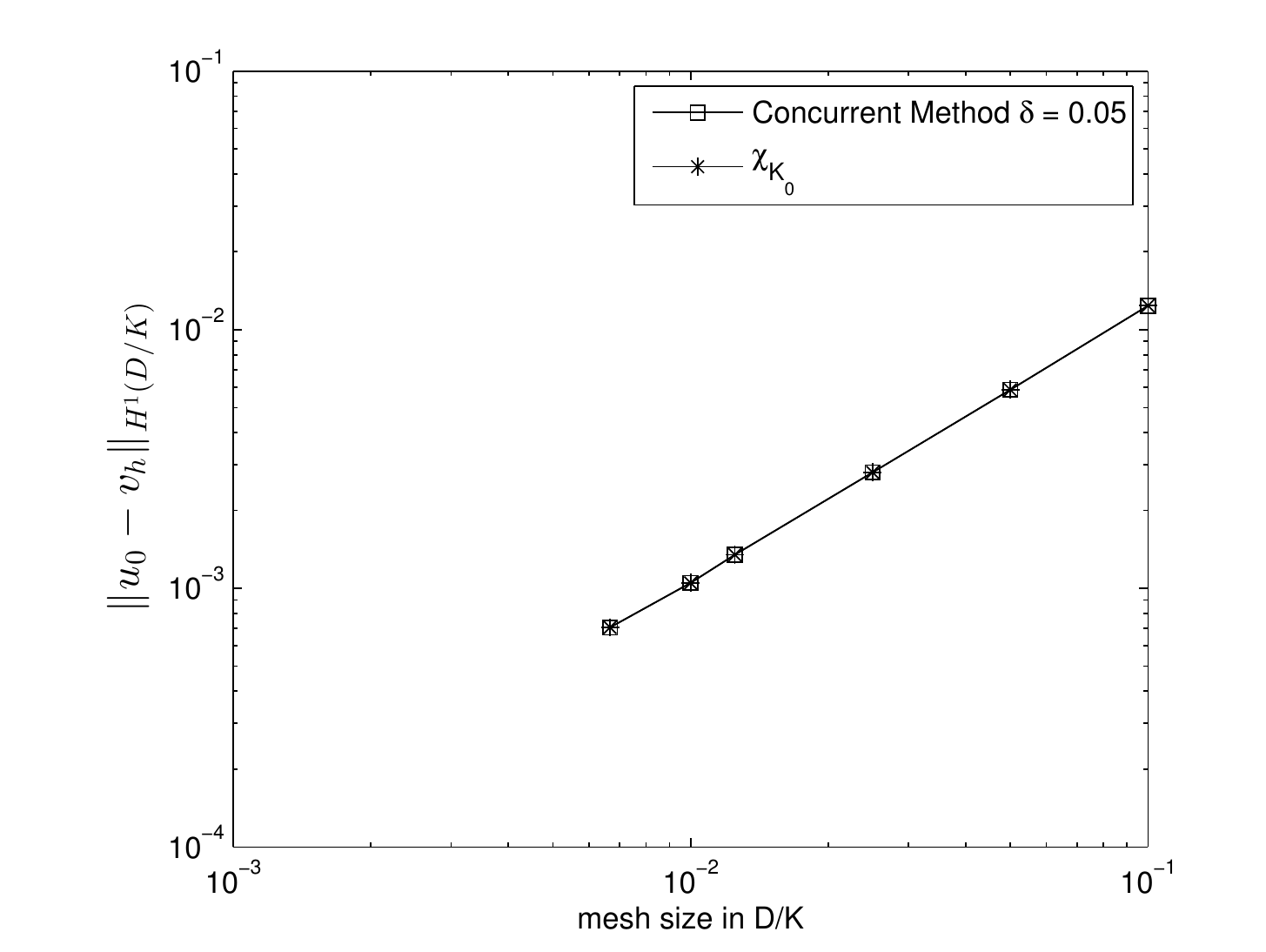}}
		\centerline{(b)}
	\end{minipage}
	\caption{(a) Localized H$^1$ error with respective to the mesh size inside $K$ with different $\delta$. (b) Error between the hybrid solution and the homogenized solution outside $K$. }
	\label{Ex2Fig1}
\end{figure}

Next we plot the error outside $K$ in Figure~\ref{Ex2Fig1}(b). The results in Table~\ref{tab:homohybridex2} shows that the convergence rate of the hybrid solution to the homogenized solution is optimal with respect to both the energy norm and the $L^2$ norm.
\begin{table}[htbp]
	\centering
	\caption{\small Error between the hybrid solution and the homogenized solution outside $K$.}\label{tab:homohybridex2}
	\begin{tabular}{|c|c|c|c|c|}\hline
		h   &$\nm{u_0-v_h}{L^2(D\backslash K)}$ &order  &$\nm{u_0-v_h}{H^1(D\backslash K)}$ &order  \\\hline
	  1/10	&	3.93E-04	&		&	1.24E-02	&		\\ \hline
	  1/20	&	8.92E-05	&	2.14	&	5.86E-03	&	1.08	\\ \hline
	  1/40	&	2.06E-05	&	2.11	&	2.81E-03	&	1.06	\\ \hline
	  1/80	&	4.92E-06	&	2.07	&	1.35E-03	&	1.06	\\ \hline
	  1/160     	&	1.45E-06	&	1.76	&	7.05E-04	&	0.94	\\ \hline	
	\end{tabular}
\end{table}
\subsection{Comparison with the global-local approach}
In this part, we compare the present method with the global-local
method~\cite{OdenVemaganti:2000}. The local region is
$\Omega_\eta=(0.5,0.5)+(-L-\eta,L+\eta)^2$ for a positive parameter
$\eta$. The recovered solution is denoted by $\wt{u}^\epsilon$. The results for Example~\ref{subsec:2ex} and Example~\ref{subsec:1ex} are plotted in
Figure~\ref{Ex1Fig2} and Figure~\ref{Ex2Fig2}, respectively. The results in both figures show that the concurrent
approach yields comparable results with those obtained by the
global-local approach, and the concurrent approach being slightly
more accurate. Moreover, it seems the parameter $\eta$ has little
effect on the
accuracy of the global-local method.
\begin{figure}[htbp]
	\begin{minipage}{0.48\linewidth}
		\centerline{\includegraphics[width=6.8cm]{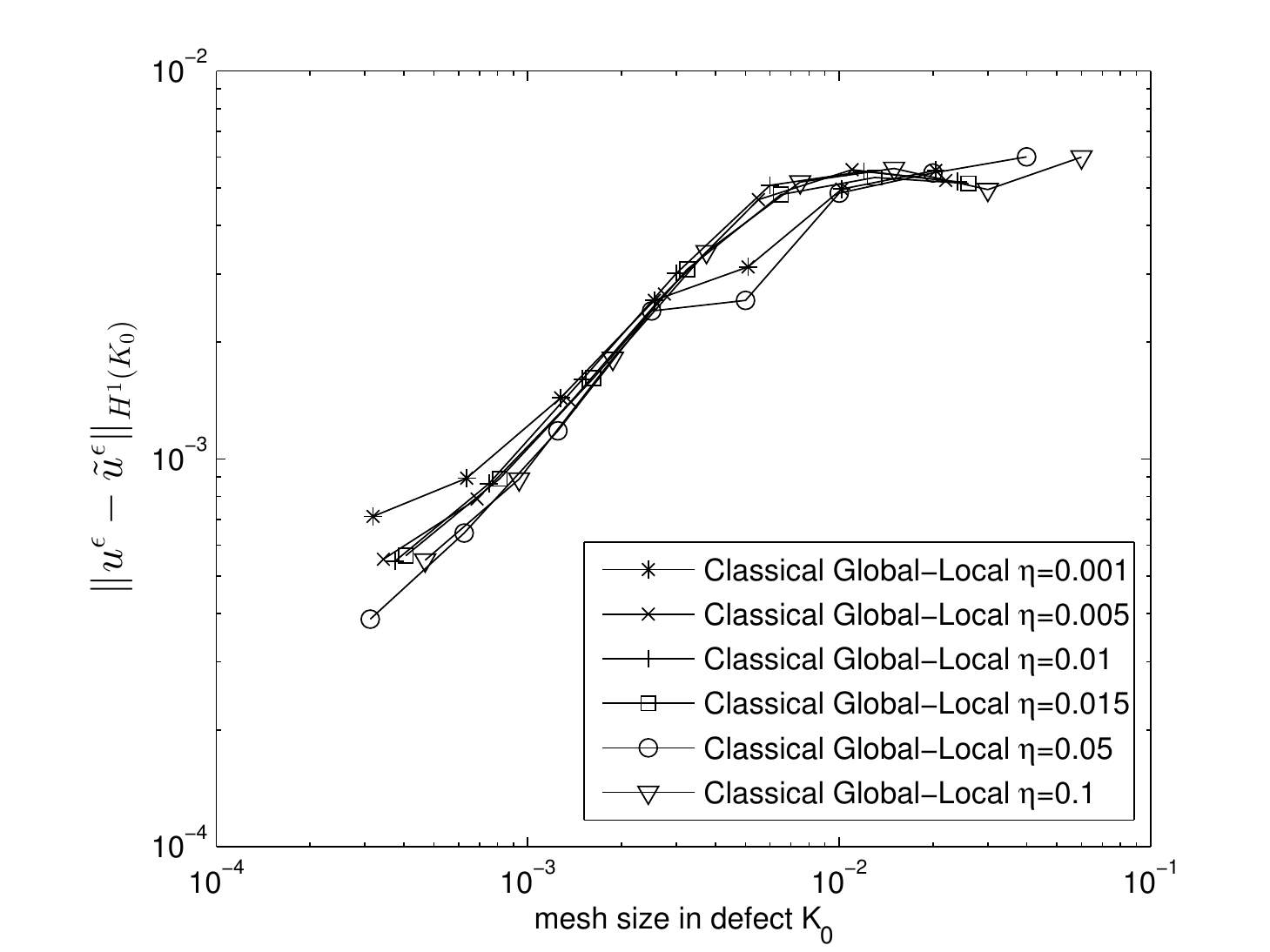}}
		\centerline{(a)}
	\end{minipage}
	\begin{minipage}{0.48\linewidth}
		\centerline{\includegraphics[width=6.8cm]{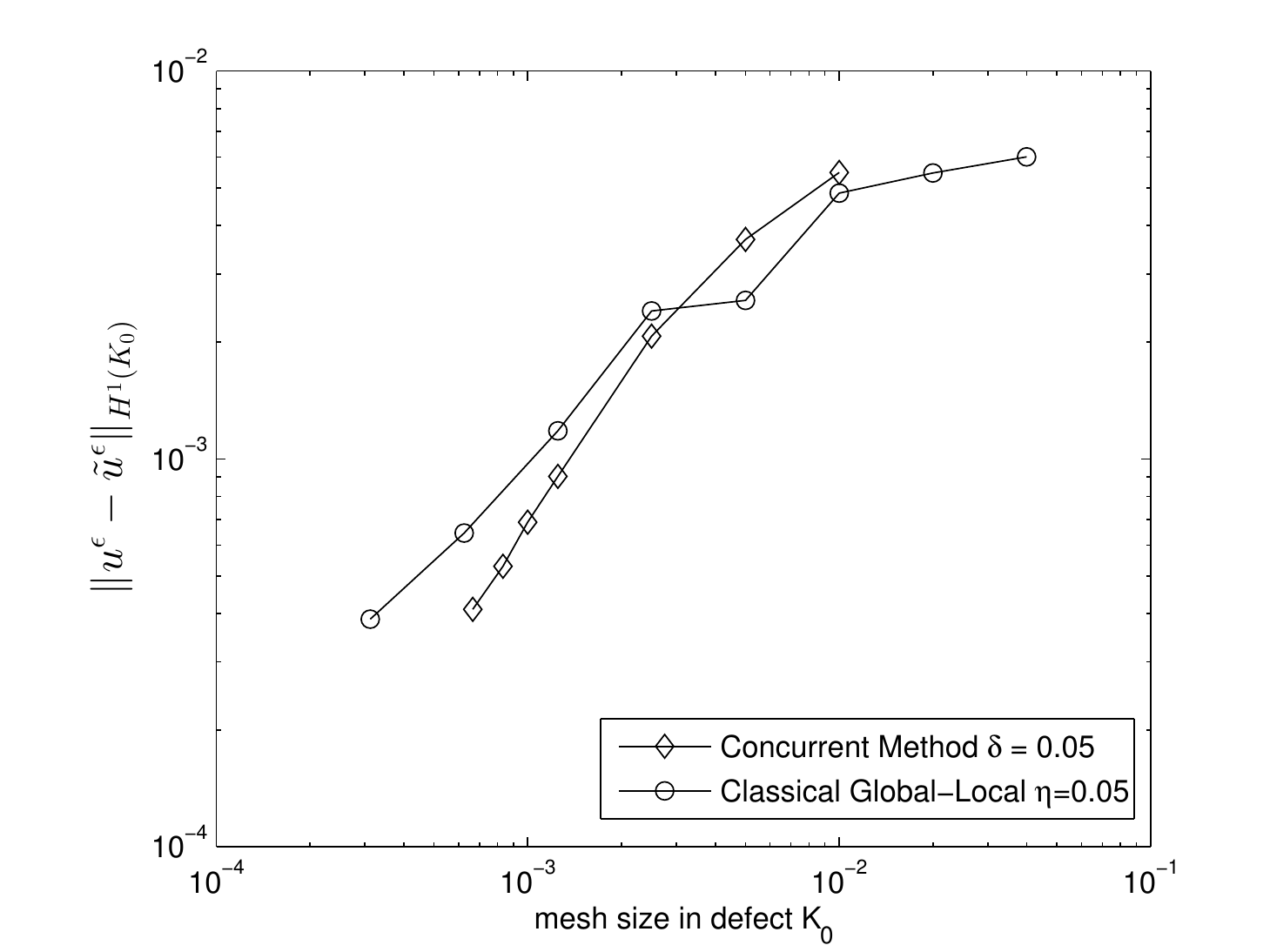}}
		\centerline{(b)}
	\end{minipage}
	\caption{(a)Localized H$^1$ error inside $K_0$ for Example~\ref{subsec:2ex} by the global-local method with different $\eta$. (b) Comparison of the localized H$^1$ error inside $K_0$ for Example~\ref{subsec:2ex} with the global-local approach and the concurrent method.}\label{Ex1Fig2}
\end{figure}
\begin{figure}[htbp]
	\begin{minipage}{0.48\linewidth}
		\centerline{\includegraphics[width=6.8cm]{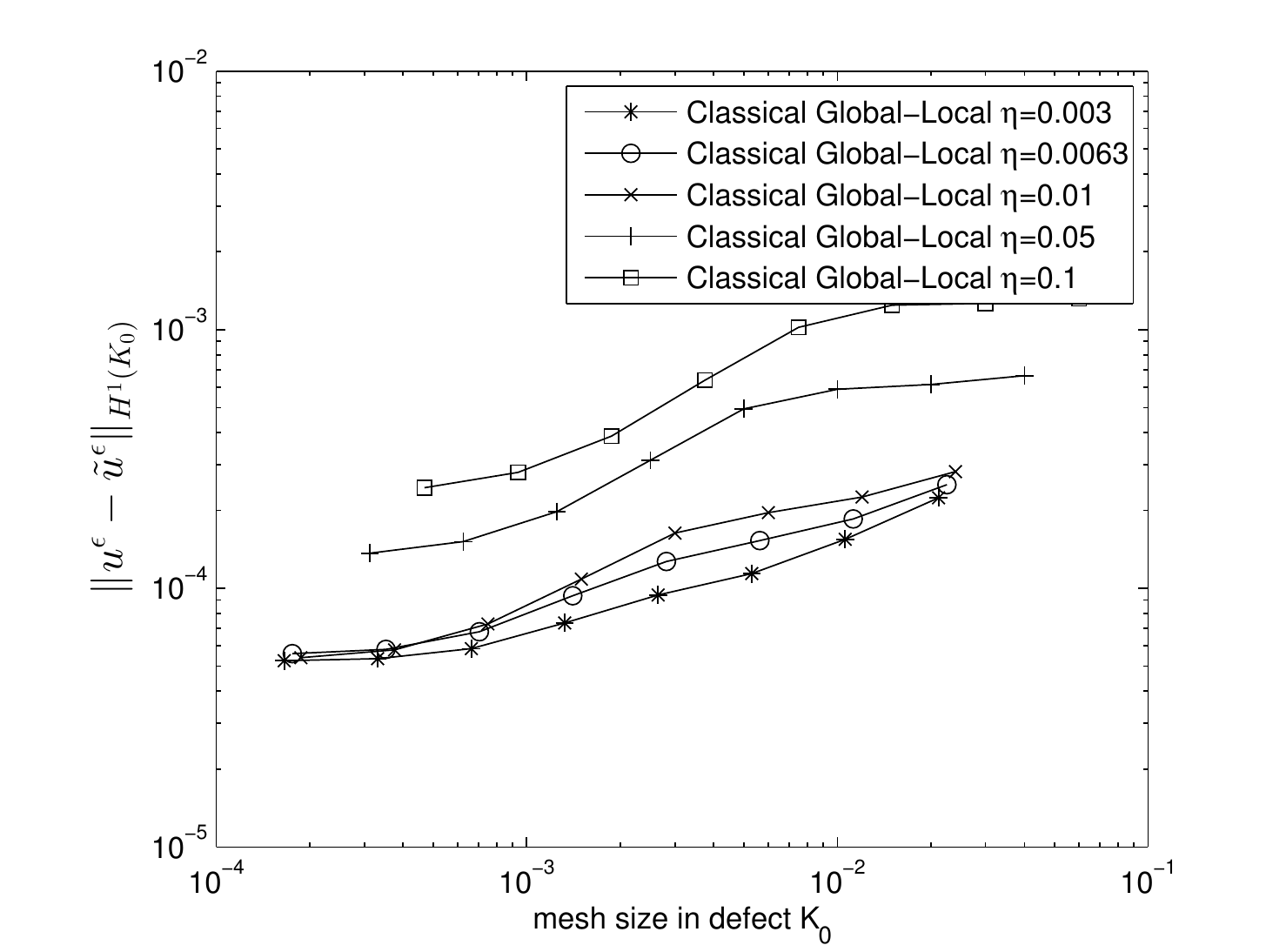}}
		\centerline{(a)}
	\end{minipage}
	\begin{minipage}{0.48\linewidth}
		\centerline{\includegraphics[width=6.8cm]{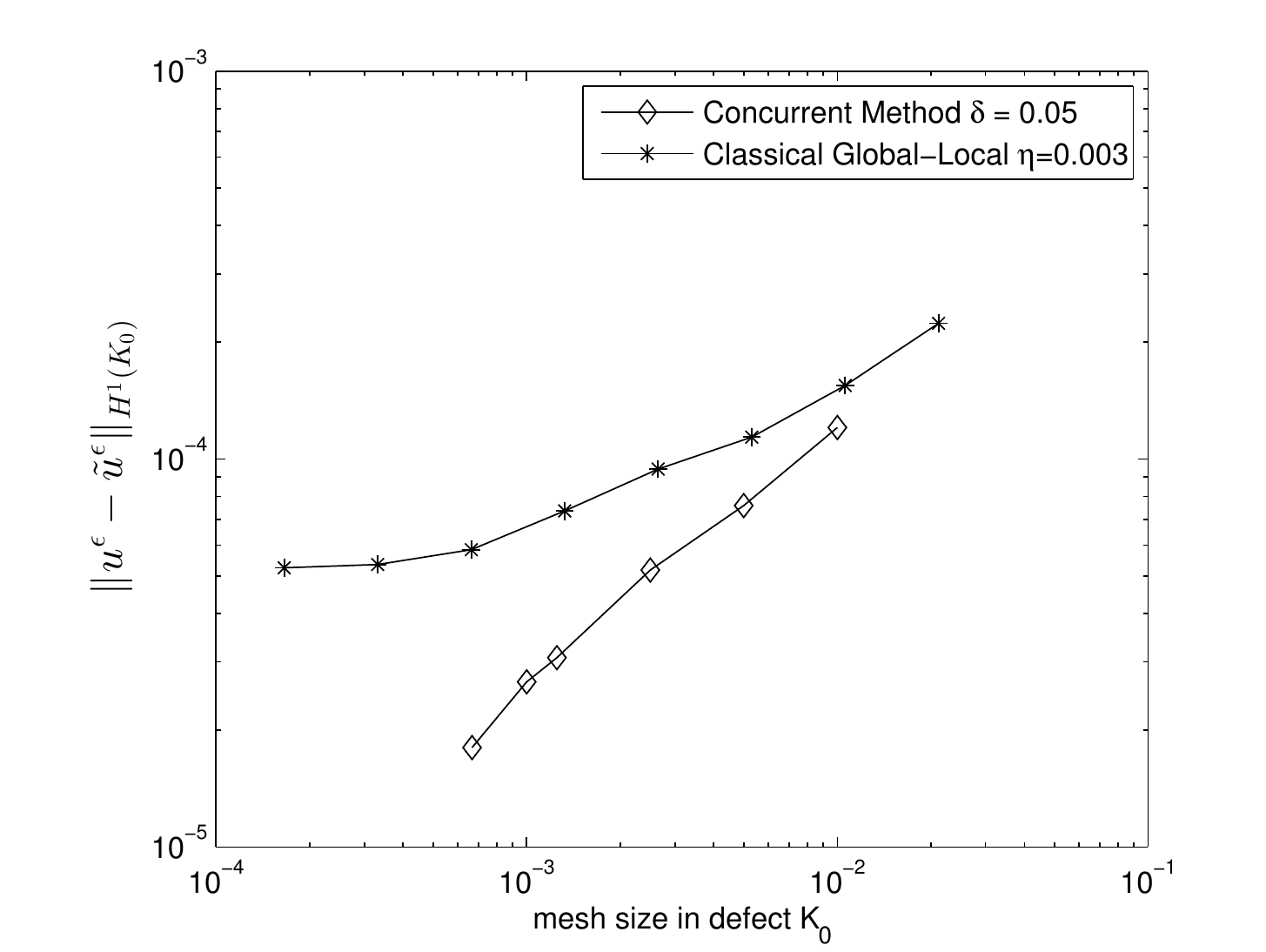}}
		\centerline{(b)}
	\end{minipage}	
	\caption{(a) Localized H$^1$ error inside $K_0$ for Example~\ref{subsec:1ex} by the global-local method with different $\eta$. (b) Comparison of the localized H$^1$ error inside $K_0$ for Example~\ref{subsec:1ex} with the global-local approach and the concurrent method.}
	\label{Ex2Fig2}
\end{figure}
%
\section{Conclusion}
We propose a new hybrid method that retrieves the global macroscopic
information and resolves the local events simultaneously. The
efficiency and accuracy of the proposed method have been demonstrated
for problems with or without scale separation. The rate of convergence
has been established when the coefficient is either periodic or
almost-periodic.

For possible future directions, the formulation of the method can be
naturally extended to treat problems with finite number of localized
defects, the random coefficients and also time-dependent problems. It
is also interesting to study the case when the local mesh inside the
defect domain is not body-fitted, which can be done with the aid of
the existing methods for elliptic interface problem; See e.g.,~\cite
{Gunzman:2016}. We shall leave these for further exploration.
%
\begin{appendix}
\section{Example}~\label{subsec:example}
To better appreciate the
estimates~\eqref{eq:diserrhomo1} and~\eqref{eq:diserrhomo2}, which are crucial in our analysis, let us consider a one-dimensional problem
\[
\left\{
\begin{aligned}
& -\Lr{\a(x)u'(x)}^{\prime}=0,\quad x\in(0,1),\\
& u(0)=0,\quad\a(1)u'(1)=1,
\end{aligned}\right.
\]
where $\a(x)=2+\sin(x/\eps)$. A direct calculation gives that the effective coefficient $\mc{A}=\sqrt3$ and the solution of the homogenized problem is $\cu(x)=x/\mc{A}$.

We consider a uniform mesh given by
\[
x_0=0<x_1=h<\cdots<x_i=ih<\cdots<x_{2N}=1,
\]
where $h=1/(2N)$. The finite element space $X_h$ is simply the piecewise linear element associated with the above mesh with zero boundary
condition at $x=0$.

\smallskip
\noindent \textit{Case $h \gg \eps$.}
We firstly consider the case that $h\gg\eps$, while the precise relation
between $h$ and $\eps$ will be made clear below. Denote $v_h(x_j)=v_j$
and the interval $I_j=(x_{j-1},x_j)$, the mean of the coefficients
$\bb$ over each $I_j$ is denoted by $b_j=\negint_{I_j}\bb(x)\dx$.

We define the transition function $\rho$ as a piecewise linear  function that is supported in $(-2L,2L)$, where $L$ is a fixed number with $0<L<1/4$. Without loss of generality, we assume that $L=Mh$ with $M$ an integer. In particular,
\[
\rho(x)=\left\{\begin{aligned}
0&\quad 0\le x\le x_{N-2M},\\
\dfrac{x-x_{N-2M}}{L}&\quad x_{N-2K}\le x\le x_{N-M},\\
1&\quad x_{N-M}\le x\le x_{N+M},\\
\dfrac{x_{N+2M}-x}{L}&\quad x_{N+M}\le x\le x_{N+2M},\\
0&\quad x_{N+2M}\le x\le x_{2N}=1.
\end{aligned}\right.
\]
By construction, we get the size of the support of $\rho$ is $\abs{K}=4L$.

We easily obtain the linear system for $\{v_j\}_{j=1}^{2N}$ as
\[
\left\{
\begin{aligned}
-b_jv_{j-1}+(b_j+b_{j+1})v_j-b_{j+1}v_{j+1}&=0,\quad j=1,\cdots,2N-1,\\
-b_{2N}v_{2N-1}+b_{2N}v_{2N}&=h.
\end{aligned}\right.
\]
Define $c_j{:}=(v_j-v_{j-1})b_j/h$, we rewrite the above equation as
\[
c_j-c_{j-1}=0,\quad j=1,\cdots,2N-1,\qquad c_{2N}=1.
\]
Hence $c_j=1$ for $j=1,\cdots,2N$, and the above linear system reduces to
\[
(v_j-v_{j-1})b_j=h.
\]
Using $v_0=0$, we obtain
\begin{equation}\label{eq:exactsolu}
v_j=h\sum_{i=1}^j\dfrac1{b_i}.
\end{equation}
Observing that $v_h(x)=\cu(x)$ for $x\in[0,x_{N-2M}]$ because they are linear functions that coincide at all the nodal points $x_i$ for $i=0,\cdots, N-2M$.

For $x\in I_{N-2M+j+1}$, we obtain
\[
\cu(x)-v_h(x)=h\sum_{i=1}^j\Lr{\dfrac1{\mc{A}}-\dfrac{1}{b_{N-2M+i}}}+(x-x_{N-2M+j})\Lr{\dfrac1{\mc{A}}-\dfrac1{b_{N-2M+j+1}}}.
\]
Define $S_j{:}=h\sum_{i=1}^j\Lr{\dfrac1{\mc{A}}-\dfrac{1}{b_{N-2M+i}}}$, we rewrite the above equation as
\begin{equation}\label{eq:func}
\cu(x)-v_h(x)=\dfrac{x_{N-2M+j+1}-x}{h}S_j+\dfrac{x-x_{N-2M+j+1}}{h}S_{j+1},
\end{equation}
which immediately yields
\begin{equation}\label{eq:examh1err}
\begin{aligned}
\int_{x_{N-2M}}^{x_{N-M}}\abs{\cu'(x)-v_h'(x)}^2\dx&=h\sum_{j=1}^M\abs{\dfrac1{\mc{A}}-\dfrac1{b_{N-2M+j}}}^2\\
&\ge\dfrac{h}{27}\sum_{j=1}^M\abs{\mc{A}-b_{N-2M+j}}^2.
\end{aligned}
\end{equation}
This is the starting point of later derivation. A direct calculation gives
\begin{align*}
b_{N-2M+j}-\mc{A}&=\negint_{I_{N-2M+j}}\rho(x)(a^\eps(x)-\mc{A})\dx\\
&=\dfrac{2-\mc{A}}{2}\bigl(\rho(x_{N-2M+j-1})+\rho(x_{N-2M+j})\bigr)+\negint_{I_{N-2M+j}}\sin\dfrac{x}\eps\dx\\
&=\dfrac{(2-\mc{A})h}{2L}(2j-1)+\negint_{I_{N-2M+j}}\sin\dfrac{x}\eps\dx,
\end{align*}
and an integration by parts yields
\begin{align*}
\negint_{I_{N-2M+j}}\sin\dfrac{x}\eps\dx&=\dfrac{2j\eps}{L}\sin\dfrac{h}{2\eps}\sin\dfrac{x_{N-2M+j-1/2}}{\eps}\\
&\quad-\dfrac{\eps}{L}\cos\dfrac{x_{N-2M+j-1}}{\eps}+\dfrac{\eps^2}{Lh}\Lr{\cos\dfrac{x_{N-2M+j-1}}{\eps}
-\cos\dfrac{x_{N-2M+j}}{\eps}}.
\end{align*}
Combining the above two equations, we obtain
\begin{equation}\label{eq:examerr}
b_{N-2M+j}-\mc{A}=\dfrac{(2-\mc{A})h}{2L}(2j-1)+\dfrac{2j\eps}{L}\sin\dfrac{h}{2\eps}\sin\dfrac{x_{N-2M+j-1/2}}{\eps}+\text{REM},
\end{equation}
where the remainder term
\[
\text{REM}{:}=-\dfrac{\eps}{L}\cos\dfrac{x_{N-2M+j-1}}{\eps}+\dfrac{\eps^2}{Lh}\Lr{\cos\dfrac{x_{N-2M+j-1}}{\eps}
-\cos\dfrac{x_{N-2M+j}}{\eps}},
\]
which can be bounded as
\begin{align*}
\abs{\text{REM}}&\le\dfrac{\eps}{L}+\dfrac{2\eps^2}{Lh}\abs{\sin\dfrac{h}{2\eps}}\abs{\cos\dfrac{x_{N-2M+j-1/2}}{\eps}}\\
&\le\dfrac{\eps}{L}+\dfrac{2\eps^2}{Lh}\dfrac{h}{2\eps}=\dfrac{2\eps}{L}.
\end{align*}
Note that
\(
\sum_{j=1}^M(2j-1)^2=M(4M^2-1)/3,
\)
and
\[
\sum_{j=1}^Mj^2\sin^2\dfrac{x_{N-2M+j-1/2}}{\eps}
\le\sum_{j=1}^Mj^2=\dfrac{1}{6}M(M+1)(2M+1).
\]
Summing up all the above estimates and using the elementary inequality
\[
(a+b+c)^2+b^2+c^2\ge \dfrac{a^2}3\quad\text{for any\quad}a,b,c\in\RR,
\]
we have, for $M\ge 3$,
\begin{align*}
\sum_{j=1}^M\abs{\mc{A}-b_{N-2M+j}}^2&\ge\dfrac1{3}\dfrac{(2-\mc{A})^2h^2}{4L^2}\sum_{j=1}^M(2j-1)^2-
\dfrac{4\eps^2}{L^2}\sin^2\dfrac{h}{2\eps}\sum_{j=1}^Mj^2\sin^2\dfrac{x_{N-2M+j-1/2}}{\eps}\\
&\quad-\dfrac{4Mh\eps^2}{L^2}\\
&\ge\dfrac{(2-\mc{A})^2h^2}{36L^2}M(4M^2-1)-\dfrac{2\eps^2}{3L^2}M(M+1)(2M+1)-\dfrac{4M\eps^2}{L^2}\\
&\ge\dfrac{(2-\mc{A})^2h^2}{36L^2}M(4M^2-1)-\dfrac{2\eps^2}{3L^2}M(4M^2-1)\\
&\ge\dfrac{(2-\mc{A})^2h^2}{72L^2}M(4M^2-1)
\end{align*}
provided that $\eps/h\le (2-\mc{A})/(4\sqrt{3})$. Substituting the above estimate into~\eqref{eq:examh1err}, we obtain
\begin{align*}
\int_{x_{N-2M}}^{x_{N-M}}\abs{\cu'(x)-v_h'(x)}^2\dx
&\ge\dfrac{(2-\mc{A})^2h^3}{1944L^2}M(4M^2-1)\\
&\ge\dfrac{(2-\mc{A})^2h^3}{648L^2}M^3=\dfrac{(2-\mc{A})^2}{648}L.
\end{align*}
This implies
\[
\nm{\cu'-v_h'}{L^2(1/2-2L,1/2-L)}\ge\dfrac{2-\mc{A}}{18\sqrt{2}}L^{1/2}=\dfrac{2-\mc{A}}{36\sqrt{2}}\abs{K}^{1/2}.
\]
This shows that the factor $\abs{K}^{1/2}$ in~\eqref{eq:diserrhomo1} is sharp.
The same argument shows the size-dependence of $\abs{K}$ in the estimate~\eqref{eq:diserrhomo2}.


\smallskip
\noindent\textit{Case $h \ll \eps$.}
We next consider the case when $h\ll\eps$. In fact, we may employ coarser mesh with mesh size $H$ outside the defect region with $H\gg h$, while a finer mesh with mesh size $h$ inside the defect region. The above derivation remains true and we still have $v_h(x)=\cu(x)$ for $x\in [0,1/2-2L]$. We start from the inequality~\eqref{eq:examh1err}. Notice that the dominant term in the expression of $b_{N-2M+j}-\mc{A}$ is the oscillatory one in~\eqref{eq:examerr}. Denote $\phi=2h/\eps$. A direct calculation gives
\begin{align*}
&\quad\sum_{j=1}^Mj^2\sin^2\dfrac{x_{N-2M+j-1/2}}\eps=\dfrac12\sum_{j=1}^Mj^2-\dfrac12\sum_{j=1}^M\cos\dfrac{x_{2N-4M+2j-1}}\eps\\
&=\dfrac1{12}M(M+1)(2M+1)\\
&\quad-\biggl\{\dfrac{M(M+1)}{4\sin(\phi/2)}\sin[(N-M)\phi]+\dfrac{M+1}{4\sin^2\phi/2}\cos[(N-M)\phi]\cos\dfrac{\phi}{2}\\
&\qquad\quad-\dfrac{\cos[(N-3M/2-1)\phi]\cos\dfrac{\phi}{2}\sin\dfrac{M+1}{2}\phi}{4\sin^3(\phi/2)}\biggr\}.
\end{align*}
We assume that
\begin{equation}\label{eq:cond1}
\sin\dfrac{\phi}{2}\ge\dfrac{5}{M}.
\end{equation}
Denote the terms in the curled bracket by $I$. Given~\eqref{eq:cond1}, using the
elementary inequalities $2x/\pi\le\sin x\le x$ for $x\in[0,\pi/2]$, we bound $I$
as
\begin{align*}
\abs{I}&\le\dfrac{(M+1)M}{4\sin(\phi/2)}+\dfrac{M+1}{4\sin^2(\phi/2)}+\dfrac{(M+1)\phi/2}{4\sin^3(\phi/2)}\\
&\le\dfrac{(M+1)M}{4\sin(\phi/2)}+\dfrac{M+1}{4\sin^2(\phi/2)}+\dfrac{(M+1)\pi}{8\sin^2(\phi/2)}\\
&\le\dfrac{M^2(M+1)}{12},
\end{align*}
which immediately yields
\[
\sum_{j=1}^Mj^2\sin^2\dfrac{x_{N-2M+j-1/2}}\eps\ge\dfrac{M^3}{12}.
\]
This implies
\[
\dfrac{4\eps^2}{L^2}\sin^2\dfrac{h}{2\eps}\sum_{j=1}^Mj^2\sin^2\dfrac{x_{N-2M+j-1/2}}\eps\ge
\dfrac{4\eps^2}{L^2}\Lr{\dfrac2{\pi}\dfrac{h}{2\eps}}^2\dfrac{M^3}{12}=\dfrac{M}{3\pi^2}.
\]
Note also
\[
\dfrac{(2-\mc{A})^2h^2}{4L^2}\sum_{j=1}^M(2j-1)^2\le\dfrac{(2-\mc{A})^2}{3}M.
\]
Combining the above two estimates, we obtain
\begin{align*}
\sum_{j=1}^M\abs{\mc{A}-b_{N-2M+j}}^2&\ge\dfrac{1}2\sum_{j=1}^M
\Lr{\dfrac{2\eps}{L}\sin\dfrac{h}{2\eps}j\sin\dfrac{x_{N-2M+j-1/2}}\eps+\dfrac{(2-\mc{A})h}{2L}(2j-1)}^2\\
&\quad-\dfrac{4M\eps^2}{L^2}\\
&\ge\Lr{\dfrac{1}{6}\Lr{1/\pi+\mc{A}-2}^2-\dfrac{4\eps^2}{L^2}}M>0
\end{align*}
provided that
\[
h>\dfrac{\sqrt{6}\eps}{(1/\pi+\mc{A}-2)M}.
\]
This condition suffices for the validity of~\eqref{eq:cond1}, which is satisfied under a weaker condition
\(
h>{5\pi\eps}/(2M).
\)

Substituting the above estimate into~\eqref{eq:examh1err}, we may find that there exists $C$ depending only on $\mc{A}$ such that
\[
\nm{\cu'-v_h'}{L^2(1/2-2L,1/2-L)}\ge CL^{1/2}=\dfrac{C}{2}\abs{K}^{1/2}.
\]
This proves that the factor $\abs{K}^{1/2}$ is sharp for~\eqref{eq:diserrhomo1}.
The same argument shows the size-dependence of $\abs{K}$ in the estimate~\eqref{eq:diserrhomo2}.
\end{appendix}
\bibliographystyle{amsxport}
\bibliography{mulspde}
\end{document}